\documentclass[12pt,reqno]{amsart}
\usepackage{graphicx}

\usepackage{enumerate}
\usepackage{amsmath, amsthm} 
\usepackage{amsfonts}
\usepackage{amssymb}
\usepackage{fullpage} 
\usepackage{xcolor}

\newtheorem{theorem}{Theorem}[subsection]

\newtheorem{conj}[theorem]{Conjecture}

\newtheorem{lemma}[theorem]{Lemma}
\newtheorem{proposition}[theorem]{Proposition}
\newtheorem{corollary}[theorem]{Corollary}

\theoremstyle{definition}
\newtheorem{definition}{Definition}[subsection]
\newtheorem{example}[definition]{Example}
\newtheorem{remark}[definition]{Remark}

\usepackage{mathabx,amsmath,amscd, latexsym, amssymb, bbold}

\usepackage{amsmath, amssymb, tikz}

\DeclareMathOperator{\ev}{ev}
\DeclareMathOperator{\coev}{coev}

\def\bC{\mathbb{C}}
\def\bZ{\mathbb{Z}}

\def\cC{\mathcal{C}}

\def\NK{{\mathcal{K}}}
\def\Kn{{\NK_n}}
\def\RKn{\widetilde{D\Kn}}
\def\H{\tilde H}

\DeclareMathOperator{\id}{id}
\DeclareMathOperator{\Hom}{Hom}
\DeclareMathOperator{\Aut}{Aut}
\DeclareMathOperator{\Rep}{Rep}
\DeclareMathOperator{\sgn}{sgn}

\renewcommand{\Vec}{\mathrm{Vec}}
\def\Vecm{\Vec_{\bZ_2}^-}
\DeclareMathOperator{\sVec}{sVec}
\DeclareMathOperator{\End}{End}

\DeclareMathOperator{\Id}{Id}

\def\bF{\mathbb{F}}
\def\incl{\hookrightarrow}
\def\rib{\tilde{\mathbf{v}}}

\newcommand{\mcA}{\mathcal{A}}

\newcommand{\mcK}{\mathcal{K}}

\newcommand{\mbbQ}{\mathbb{Q}}

\title{On the formal ribbon extension of a quasitriangular Hopf algebra}
\date{}

\author[ucsb]{Quinn T. Kolt}
\email{quinn@math.ucsb.edu}

\address{Department of Mathematics, University of California, Santa Barbara, CA 93106, USA}

\begin{document}
\begin{abstract}
    Any finite-dimensional quasitriangular Hopf algebra $H$ can be formally extended to a ribbon Hopf algebra $\H$ of twice the dimension. We investigate this extension and its representations. We show that every indecomposable $H$-module has precisely two compatible $\H$-actions. We investigate the behavior of simple, projective, and M\"uger central $\H$-modules in terms of these $\H$-actions. We also observe that, in the semisimple case, this construction agrees with the pivotalization/sphericalization construction introduced by Etingof, Nikshych, and Ostrik (2003). As an example, we investigate the formal ribbon extension of odd-index doubled Nichols Hopf algebras $D\Kn$. 
\end{abstract}
\maketitle

\tableofcontents

\section{Introduction}
    Reshetikhin and Turaev \cite{ReshetikhinTuraev1990} introduced the notion of a ribbon Hopf algebra and showed that any quasitriangular Hopf algebra $H$ can be formally extended to a ribbon Hopf algebra $\H$ of twice the dimension. These formal ribbon extensions are further discussed in \cite{Andruskiewitsch2014hopftensor} wherein Sommerh\"auser remarks that $\H$ can be factored as a cocycled crossed product of $H$ and $\bC[\bZ_2]$. This article seeks to study these extensions and their representations. Specifically, we seek a concrete understanding of the tensor category $\Rep(\H)$ in terms of $\Rep(H)$. Our interest lies in generating examples of both semisimple and non-semisimple ribbon categories, as well as better understanding braided tensor categories which are not ribbon.

    The question of whether every braided fusion category is ribbon remains open. This is a special case of the question of whether every fusion category has a spherical structure \cite{eno2005fusion}. Etingof, Nikshych, and Ostrik \cite{eno2005fusion} showed that every fusion category $\cC$ embeds into a spherical fusion category of twice the dimension $\tilde\cC$. By Theorem \ref{pivotalization-agrees}, this construction agrees with the formal ribbon extension explored here, i.e., $\Rep(\H)\cong\widetilde{\Rep(H)}$ for semisimple $H$. Consequently, we may interpret the formal ribbon extension as a generalization of sphericaalization to the non-semisimple braided case.

    Topological quantum field theories (TQFTs) provide another motivation for studying ribbon Hopf algebras. Reshetikhin-Turaev and Crane-Yetter-Hauffman TQFTs are semisimple (2+1)- and (3+1)-TQFTs which come from certain ribbon Hopf algebras (and more generally ribbon fusion categories) \cite{Crane1997statesum,Reshetikhin1991invariants,turaev1992modular}. Recently, there has been much interest in non-semisimple TQFTs. One potential advantage of studying non-semisimple TQFTs is that semisimple (3+1)-TQFTs are known to be unable to distinguish exotic smooth structure \cite{Reutter2023semisimple}. It is unknown whether this is also true for non-semisimple (3+1)-TQFTs. Like in the semisimple case, one build non-semisimple (2+1)- and (3+1)-TQFTs with  non-semisimple ribbon categories \cite{Costantino:2023bjb,DeRenzi2022tqft,kerler2003homology}. Consequently, generating examples of ribbon Hopf algebras, as we do here, may have interesting applications to both topology and physics.
    
    The present article is ordered as follows; in Section \ref{formal-ribbon-sec}, we study the basic properties of $\H$, including expanding on two known factorizations of $\H$ in terms of $H$ and $\bF[\bZ_2]$, where $\bF$ is the underlying field \cite{Andruskiewitsch2014hopftensor,ReshetikhinTuraev1990}. It follows from Sommerh\"auser's factorization that $\Rep(\H)$ fits into an exact sequence $\Vec_{\bZ_2}\to\Rep(\H)\to\Rep(H)$ of braided finite tensor categories. However, we see in Section \ref{rep-section} that $\Rep(\H)$ can be described more precisely. In Section \ref{rep-section}, we show that every $H$-module has a compatible $\H$-action, using the holomorphic functional calculus and the model completeness of the theory of algebraically closed fields. We spend the majority of this section investigating this $\H$-action. As mentioned, we also show that, if $H$ is semisimple, $\Rep(\H)$ is isomorphic to the pivotalization/sphericalization $\widetilde{\Rep (H)}$ of $\Rep (H)$ introduced in \cite{eno2005fusion}. In Section \ref{DKn-section}, we investigate the formal ribbon extension of the Drinfeld doubles of Nichols Hopf algebras\footnote{Nichols Hopf algebras are distinct from Nichols algebras of a braided vector space. However, Nichols Hopf algebras may be constructed from Nichols algebras. The terminology ``Nichols Hopf algebra'' appears in \cite{EGNO}. The algebras $D\Kn$ for even $n$ are also known as symplectic fermion ribbon Hopf algebras \cite{farsad2022symplectic}.} of odd index. Doubled Nichols Hopf algebras $D\Kn$ are always quasitriangular but ribbon if and only if $n$ is even. Thus, for odd $n$, the algebras $D\Kn$ provide a nontrivial example of our theory.

    For simplicity, we build our theory for Hopf algebras. However, most of the theory (in particular, all of Section \ref{rep-section}) explored here can easily be extended to quasitriangular weak Hopf algebras, and, hence, to braided fusion categories.

\section{Formal ribbon extensions}\label{formal-ribbon-sec}

\subsection{Notations and terminology}
    We fix an algebraically closed field $\bF$ of characteristic 0. Throughout this article, we make use of sum-less Sweedler's notation (e.g., $\Delta(x) = x^{(1)}\otimes x^{(2)}$). Similarly, if $R\in H\otimes H$ is an $R$-matrix of a quasitriangular Hopf algebra $H$, we write $R = R^{(1)}\otimes R^{(2)}$. By an $H$-module, we always mean a left $H$-module. We denote the category of finite-dimensional representations of a Hopf algebra $H$ by $\Rep(H)$ and identify it with the category of finite-dimensional $H$-modules in the standard way. For an element $h\in H$ and an $H$-module $M\in\Rep(H)$, we denote by $h\cdot -$ the map $M\to M$ given by $m\mapsto h\cdot m$. 
    
    Let $H=(H, m, 1, \Delta, \epsilon, S)$ be a Hopf algebra. Then, $H$ is:
    \begin{enumerate}
        \item \textit{quasitriangular} if there is an invertible \textit{$R$-matrix} $R\in H\otimes H$ such that
        \begin{align*}
            R\Delta(x)R^{-1} &= \Delta^{\mathrm{op}}(x), & (\Delta\otimes\id_H)(R) &= R_{13}R_{23}, & (\id_H\otimes\Delta)(R) &= R_{13}R_{12},
        \end{align*}
        where $R_{12} = R^{(1)}\otimes R^{(2)}\otimes 1$, $R_{13} = R^{(1)}\otimes 1\otimes R^{(2)}$, and $R_{23} = 1\otimes R^{(1)}\otimes R^{(2)}$;
        \item \textit{ribbon} if $H$ is quasitriangular and there is an invertible central \textit{ribbon element} ${\mathbf{v}}\in Z(H)$ such that
        \begin{align*}
            {\mathbf{v}}^2 &= uS(u), & \Delta({\mathbf{v}}) &= (R_{21}R)^{-1}({\mathbf{v}}\otimes{\mathbf{v}}), & S({\mathbf{v}}) &= {\mathbf{v}}, & \epsilon({\mathbf{v}}) &= 1,
        \end{align*}
        where $u = S(R^{(2)})R^{(1)}$ is the \textit{Drinfeld element} and $R_{21} = R^{(2)}\otimes R^{(1)}$;
        \item \textit{unimodular} if the space of left integrals $\{\Lambda\in H|h\Lambda = \epsilon(h)\Lambda,\forall h\in H\}$ and the space of right integrals $\{\Lambda\in H|\Lambda h = \epsilon(h)\Lambda,\forall h\in H\}$ coincide;
        \item \textit{factorizable} if $H$ is quasitriangular and the Drinfeld map $f_Q:H^*\to H$ given by $f_Q(\beta) = (\beta\otimes\id_H)(R_{21}R)$ is a linear isomorphism.
    \end{enumerate}

    The Drinfeld element $u$ has a handful of nice properties that we exploit throughout this article.
    \begin{proposition}\label{Drinfeld-props}
        Let $H$ be a quasitriangular Hopf algebra. Then, the Drinfeld element $u$ satisfies
        \begin{itemize}
            \item $uS(u)\in Z(H)$
            \item $\Delta(uS(u)) = (R_{21}R)^{-2}(uS(u)\otimes uS(u))$,
            \item $uS(u)^{-1}$ is grouplike,
            \item $S^2(h) = uhu^{-1}$ for all $h\in H$,
            \item $u^{-1} = R^{(2)} S^2(R^{(1)})$.
        \end{itemize}
    \end{proposition}
    
    \begin{theorem}
        Let $H$ be a finite-dimensional Hopf algebra.
        \begin{enumerate}
            \item Quasitriangular structures on $H$ are in one-to-one correspondence with braidings on $\Rep(H)$. Given an $R$-matrix $R\in H\otimes H$, the corresponding braiding $\beta:\otimes\to \otimes^{\mathrm{op}}$ on $\Rep(H)$ is generated, for $H$-modules $M, M'$, by $\beta_{M,M'}:m\otimes m'\mapsto [R^{(2)}\cdot m']\otimes [R^{(1)}\cdot m]$.
            \item Ribbon elements in $H$ are in one-to-one correspondence with ribbon structures on $\Rep(H)$. Given a ribbon element $\mathbf{v}\in H$, the corresponding ribbon structure $\theta:\Id_{\Rep(H)}\to \Id_{\Rep(H)}$ on $\Rep(H)$ is given, for an $H$-module $M$, by $\theta_{M}:m\mapsto \mathbf{v}^{-1}\cdot m$.
            \item $H$ is unimodular if and only if $\Rep(H)$ is unimodular.
            \item $H$ is factorizable if and only if $\Rep(H)$ is factorizable.
        \end{enumerate}
    \end{theorem}
    
    We denote by $\Vecm$ the ribbon fusion category whose underlying fusion category is the $\bZ_2$-graded vector spaces $\Vec_{\bZ_2}$ with simple representatives $V^+$ and $V^-$, whose braiding $\beta:\otimes\to \otimes^{\mathrm{op}}$ is trivial and nontrivial twist $\theta:\Id_{\Vecm}\to\Id_{\Vecm}$ satisfies $\theta_{V^-} = -\id_{V^-}$. In particular, if $\bF[\bZ_2]$ is the ribbon Hopf algebra with $R$-matrix $R=1\otimes 1$ and ribbon element $\mathbf{v}=g$, where $\langle g\rangle = \bZ_2$, then $\Vecm\cong \Rep(\bF[\bZ_2])$.

    \subsection{The formal ribbon extension $\H$}
    Ribbon Hopf algebras were introduced in \cite{ReshetikhinTuraev1990} to construct invariants of links. To prove the versatility of their construction, they showed that any quasitriangular Hopf algebra $H$ embeds into a ribbon Hopf algebra with the same $R$-matrix. In this paper, we term this ribbon Hopf algebra the \textit{formal ribbon extension of $H$}. As seen in Definition \ref{formal-ribbon}, the construction is quite straightforward. As we will see throughout the present article, it is also quite well-behaved.
    \begin{definition}\label{formal-ribbon}
        Let $H$ be a quasitriangular Hopf algebra with $R$-matrix $R\in H\otimes H$. Let $u = (m\circ(S\otimes \id))(R_{21})$ be the Drinfeld element. The \textit{formal ribbon extension} of $(H, R)$ is the ribbon Hopf algebra $\H$ defined by adjoining a formal ribbon element $\rib\in \H$ as follows; as a vector space $\H = H\oplus H\rib$, and $\rib\in Z(\H)$ is defined to satisfy the following identities:
        \begin{align*}
            \rib^2 &= uS(u), & \Delta(\rib) &= (R_{21}R)^{-1}(\rib\otimes\rib), & S(\rib) &= \rib, & \epsilon(\rib) &= 1.
        \end{align*}
    \end{definition} 
    
    We now review some elementary properties of this construction.
    \begin{proposition}
        Let $H$ be a finite-dimensional, quasitriangular Hopf algebra. Then, $H$ is unimodular if and only if $\H$ is unimodular.
    \end{proposition}
    \begin{proof}
        Let $\mathrm{int}_L(H)$ and $\mathrm{int}_R(H)$ be the space of left and right integrals of $H$ respectively. We claim $\mathrm{int}_L(\H) = (1+\rib)\mathrm{int}_L(H)$ and $\mathrm{int}_R(\H)=(1+\rib)\mathrm{int}_R(H)$.
        
        Let $\Lambda$ be a nonzero left integral of $H$. Then, $(1+\rib)\Lambda$ is also nonzero since $H\cap H\rib=\{0\}$. Moreover, since $1+\rib\in Z(\H)$, we have
        $$(a+b\rib)(1+\rib)\Lambda = \epsilon(a)(1+\rib)\Lambda+\epsilon(b)(\epsilon(\rib^2)+\rib)\Lambda = (\epsilon(a)+\epsilon(b))(1+\rib)\Lambda=\epsilon(a+b\rib)(1+\rib)\Lambda.$$
        Thus, $(1+\rib)\Lambda$ is a left integral. Since $\mathrm{int}_L(H)$ is one-dimensional for any finite-dimensional Hopf algebra $H$, it follows $\mathrm{int}_L(\H) = (1+\rib)\mathrm{int}_L(H)$. The proof of the statement for right integrals is similar. It follows that, if $H$ is unimodular, so is $\H$. By the linear independence of $H$ and $H\rib$, the converse is also true.
    \end{proof}

    While the formal ribbon extension preserves unimodularity, it cannot preserve factorizability, as a consequence of Proposition \ref{not-factorizable}. 
    \begin{proposition}\label{not-factorizable}
        Let $(H, R)$ be a quasitriangular Hopf algebra, and let $H\subsetneq H'$ be a Hopf algebra extension for which $(H', R)$ is quasitriangular. Then, $(H', R)$ is not factorizable.
    \end{proposition}
    \begin{proof}
        Since $R\in H\otimes H$, the Drinfeld map is zero on any functional which vanishes on $H$.
    \end{proof}
    \begin{remark}\label{build-modular}
       While $\H$ cannot be factorizable, we may build factorizable, ribbon Hopf algebras using the formal ribbon extension and the Drinfeld double constructions. If $H$ is a finite-dimensional, quasitriangular, unimodular Hopf algebra, then $D\H$ is a ribbon Hopf algebra \cite{cohen2008characters}. In particular, if $H$ is any finite-dimensional Hopf algebra, then $D\widetilde{DH}$ is a ribbon Hopf algebra. This implies that, from any finite-dimensional Hopf algebra $H$, we may construct a (possibly non-semisimple) modular category $\Rep(D\widetilde{DH})$, which has dimension $4(\dim H)^4$.
    \end{remark} 

    \subsection{Decompositions of $\H$}
    We consider two decompositions of $\H$ in terms of $H$ and $\bF[\bZ_2]$ and the implications of these decompositions on their representation categories. First, we show that $\H$ factors as a tensor product of $H$ and $\bF[\bZ_2]$ if and only if $H$ is ribbon.  This tensor product decomposition is first noted in Reshetikhin and Turaev's original paper \cite[Rmk. 3.5]{ReshetikhinTuraev1990}, though this equivalence of conditions appears to be new. Sommerh\"auser also provides a decomposition for any Hopf algebra $H$, which says that $\H$ always factors as a cocycled crossed product of $H$ and $\bF[\bZ_2]$. It is worth noting that, even if $H$ is ribbon, the cocycle may be nontrivial. This is seen in Example \ref{decompositions-differ}.

    \begin{proposition}\label{ribbon-decomp}
        Let $(H, R, {\mathbf{v}})$ be a ribbon Hopf algebra and $(\H, R, \rib)$ be the formal ribbon extension of $(H, R)$. Then, $\H \cong H\otimes \bF[\bZ_2]$ as ribbon Hopf algebras. Moreover, if $H$ is a finite-dimensional, quasitriangular Hopf algebra and $\H\cong H\otimes \bF[\bZ_2]$ as quasitriangular Hopf algebras, then $H$ has a compatible ribbon element.
    \end{proposition}
    \begin{proof}
        Denote the generator of $\bZ_2$ by $g$.  Namely, we define $\phi:\H\to H\otimes \bF[\bZ_2]$ by
        $$\phi(a + b\rib) := a\otimes 1 + b{\mathbf{v}}\otimes g.$$
        This is clearly a linear isomorphism (since ${\mathbf{v}}$ is invertible) and preserves $1$. Under this isomorphism, $g$ can be identified with $1\otimes g = \phi({\mathbf{v}}^{-1}\rib)$. Moreover,
        $$\phi((a+b\rib)(a'+b'\rib)) = (aa'+bb'uS(u))\otimes 1 + (ab'+ba'){\mathbf{v}}\otimes g = \phi(a+b\rib)\phi(a'+b'\rib).$$
        As a result of $\phi$ being an algebra homomorphism, it suffices to check that $\phi$ preserves the counit, comultiplication, and antipode on just $\rib$:
        \begin{align*}
            (\phi\otimes \phi)(\Delta(\rib)) &= (\phi\otimes \phi)((R_{21}R_{12})^{-1}(\rib\otimes\rib)) \\
            &= ((R_{21}R_{12})^{-1})({\mathbf{v}}\otimes 1\otimes{\mathbf{v}}\otimes 1)(1\otimes g\otimes 1\otimes g)\\
            &= \Delta(\phi(\rib)),\\
            \phi(S(\rib)) &= \phi(\rib) = {\mathbf{v}}\otimes g = S({\mathbf{v}}\otimes g) = S(\phi(\rib)),\\
            \epsilon(\phi(\rib)) &= \epsilon(g)\epsilon({\mathbf{v}}) = 1 = \epsilon(\rib).
        \end{align*}
        Clearly, $(R^{(1)}\otimes 1)\otimes (R^{(2)}\otimes 1)$ is an $R$-matrix for $H\otimes \bF[\bZ_2]$, and ${\mathbf{v}}\otimes g$ is a ribbon element for $H\otimes \bF[\bZ_2]$. 

        For the converse, note that, if $\Rep(H\otimes \bF[\bZ_2])$ has a ribbon structure $\theta:\Id_{\Rep(H\otimes \bF[\bZ_2])}\to \Id_{\Rep(H\otimes \bF[\bZ_2])}$, then we could restrict $\theta$ to a ribbon structure on modules of the form $M\otimes \bF_{\text{triv}}$, which is precisely $\Rep(H)$. In particular, the ribbon element $\mathbf{v}\in H$ can be recovered from the equality $\theta_{H\otimes\bF_{\text{triv}}}(1\otimes 1) = \mathbf{v}^{-1}\otimes 1$.
    \end{proof}

    \begin{corollary}\label{ribbon-reps-factor}
        Let $(H, R, {\mathbf{v}})$ be a finite-dimensional ribbon Hopf algebra and $(\H, R, \rib)$ be the formal ribbon extension of $(H, R)$. Then, 
        $$\Rep(\H)\cong \Rep(H)\boxtimes \Vecm$$
        as ribbon finite tensor categories.
    \end{corollary}

    Corollary \ref{ribbon-reps-factor} is reminiscent of the property that, for a finite-dimensional factorizable Hopf algebra $H$, representations of the Drinfeld double $DH$ factor as a Deligne product in terms of the representations of $H$ (see \cite{shimizu2019non}):
    $$\Rep(DH)\cong \Rep(H)\boxtimes\Rep(H)^{\mathrm{op}}.$$

    The notion of a cocycled crossed product algebra comes from \cite{BlattnerCohen1986}. Conditions under which this construction yields a Hopf algebra were studied in \cite{Agore2013}. This construction can be used to describe formal ribbon extensions in general \cite{Andruskiewitsch2014hopftensor}.
    \begin{definition}\label{weak}
        Let $K$ be a Hopf algebra and $H$ a unital associative algebra. A \textit{weak action of $K$ on $H$} is a linear map $\cdot: K\otimes H\to H$ such that, for all $h, h_1, h_2\in H$ and $k, k_1, k_2\in K$,
        \begin{enumerate}
            \item $k\cdot (h_1 h_2) = (k^{(1)}\cdot h_1)(k^{(2)}\cdot h_2)$,
            \item $k\cdot 1 = \epsilon(k)1$,
            \item $1\cdot h = h$.
        \end{enumerate}
        A weak action is \textit{symmetric} if, for all $h\in H$ and $k\in K$, 
        $$k^{(1)}\otimes k^{(2)}\cdot h = k^{(2)}\otimes k^{(1)}\cdot h.$$ 
    \end{definition}
    \begin{definition}\label{cocycle}
        Suppose $K$ acts weakly on $H$ via $\cdot:K\otimes H\to H$. A linear map $\sigma:K\otimes K\to H$
        \begin{enumerate}
            \item is \textit{normal} if, for all $k\in K$, $\sigma(k\otimes 1) = \sigma(1\otimes k) = \epsilon(k)1$;
            \item is a \textit{cocycle} if, for all $k_1, k_2, k_3\in K$, 
            $$k_1^{(1)}\cdot \sigma(k_2^{(1)}\otimes k_3^{(1)})\sigma(k_1^{(2)}\otimes (k_2^{(2)}k_3^{(2)})) = \sigma(k_1^{(1)}\otimes k_2^{(1)})\sigma((k_1^{(2)}k_2^{(2)})\otimes k_3);$$
            \item satisfies the \textit{twisted module condition} if, for all $k_1, k_2\in K$ and $h\in H$,
            $$k_1^{(1)}\cdot (k_2^{(1)}\cdot h)\sigma(k_1^{(2)}\otimes k_2^{(2)}) = \sigma(k_1^{(1)}\otimes k_2^{(1)})((k_1^{(2)}k_2^{(2)})\cdot h),$$
            \item is \textit{symmetric} if, for all $k_1, k_2\in K$,
            $$k_1^{(1)}k_2^{(1)}\otimes \sigma(k_1^{(2)}\otimes k_2^{(2)})=k_1^{(2)}k_2^{(2)}\otimes \sigma(k_1^{(1)}\otimes k_2^{(1)}).$$
        \end{enumerate}
    \end{definition}
    \begin{lemma}[{\cite[Lem. 4.4--5]{BlattnerCohen1986}}]\label{crossed-algebra}
        Suppose $\cdot: K\otimes H\to H$ is a weak action of a Hopf algebra $K$ on a unital associative algebra $H$. Let $\sigma:K\otimes K\to H$ be a linear map and denote by $H\#_\sigma K$ the vector space $H\otimes K$ with simple tensors denoted $h\# k$. Define a multiplication on $H\#_\sigma K$ as follows: for any $h_1, h_2\in H$ and $k_1, k_2\in K$, 
        $$(h_1\# k_1)(h_2\# k_2) = [h_1 (k_1^{(1)}\cdot h_2)\sigma(k_1^{(2)}\otimes k_2^{(1)})]\# [k_1^{(3)}k_2^{(2)}].$$
        Then, this multiplication defines a unital algebra structure on $H\#_\sigma K$ if and only if $\sigma$ is a normal cocycle with the twisted module condition.
    \end{lemma}
    In the case where the cocycle is trivial, (i.e. $\sigma(k_1\otimes k_2) = \epsilon(k_1 k_2)1_H$) and the weak action is an action, this agrees with the more classical notion of a \textit{crossed product algebra} $H\rtimes K=H\#_\sigma K$.
    \begin{lemma}[{\cite[Ex. 2.5(2)]{Agore2011extending}}]\label{crossed-hopf}
        Suppose $\cdot: K\otimes H\to H$ is a weak action of a Hopf algebra $K$ on another Hopf algebra $H$. Let $\sigma:K\otimes K\to H$ be a normal cocycle with the twisted module condition. Then, $H\#_\sigma K$ is a Hopf algebra if and only if $\cdot$ and $\sigma$ are both coalgebra homomorphisms and symmetric. The coalgebra structure is given by the tensor product of the coalgebras $H$ and $K$ and the antipode is given by 
        $$S(h\# k) = [S_H(\sigma(S_K(k^{(2)})\otimes k^{(3)}))\# S_K(k^{(1)})][S_H(h)\# 1].$$
    \end{lemma}
    
    Note that $H\cong H\#_\sigma 1$ always appears as a normal subalgebra of $H\#_\sigma K$. However, when the cocycle $\sigma$ is nontrivial, $K$ need not be a subalgebra of $H\#_\sigma K$. However, these algebras fit into a cleft exact sequence of Hopf algebras as $H\incl H\#_\sigma K\twoheadrightarrow K$. In fact, all cleft exact sequences arise this way. With this notion of cocycled crossed product, Sommerh\"auser provides a general decomposition for $\H$. However, to the author's knowledge, a proof has not been publicly shared in the literature. We present a proof here for completeness. 
    \begin{theorem}[{\cite[Rmk. 3.14 due to Sommerh\"auser]{Andruskiewitsch2014hopftensor}}]
    \label{sommer}
        Let $H$ be a finite-dimensional quasitriangular Hopf algebra and let $\bZ_2=\langle g\rangle$. Then, $\H\cong H\#_\sigma \bF[\bZ_2]$ as quasitriangular Hopf algebras, where the weak action is generated by $g\cdot h = S^2(h)$ and the cocycle $\sigma$ is generated by $\sigma(g\otimes g) = uS(u)^{-1}$.
    \end{theorem}
    \begin{proof}
        It is clear that $\cdot$ is a weak action and a coalgebra homomorphism since $S^2$ is a bialgebra homomorphism. First, we verify $\sigma$ satisfy all the conditions of Definition \ref{cocycle}. Note that normality is true by definition of $\sigma$. When verifying the cocycle and twisted module conditions, we need only check the case where $k_1=k_2=k_3=g$: for any $h\in H$,
        \begin{align*}
            g\cdot (g\cdot h)\sigma(g\otimes g)=S^4(h)uS(u)^{-1} &= uS(u)^{-1}h = \sigma(g\otimes g)(g^2\cdot h),\\
            g\cdot \sigma(g\cdot g)\sigma(g\otimes g^2) = S^2(uS(u)^{-1}) &= uS(u)^{-1} = \sigma(g\otimes g)\sigma(g^2\otimes g).
        \end{align*}
        Moreover, $\sigma$ is indeed a coalgebra homomorphism because $uS(u)^{-1}$ is grouplike. Finally, $\cdot$ and $\sigma$ are symmetric because $\bF[\bZ_2]$ is cocommutative. By Lemma \ref{crossed-hopf}, $H\#_\sigma \bF[\bZ_2]$ is a well-defined Hopf algebra.

        Observe that $\H=H\oplus H\rib^{-1}$. Let $\phi:\H\to H\#_\sigma \bF[\bZ_2]$ be given by
        $$\phi(a + b\rib^{-1}) := a\# 1 + bu^{-1}\# g$$
        for $a, b\in H$. This is clearly a linear isomorphism. It is easy to verify that this is a coalgebra homomorphism and preserves $1$. Moreover,
        \begin{align*}
            \phi(\rib^{-1})\phi(\rib^{-1}) &= (u^{-1}\# g)(u^{-1}\# g) = [u^{-1}(g\cdot u^{-1})\sigma(g\otimes g)]\# g^2 = (uS(u))^{-1}\# 1 = \phi(\rib^{-2}),\\
            S(\phi(\rib^{-1})) &= S(u^{-1}\# g) = [S_H(\sigma(S_{\bF[\bZ_2]}(g)\otimes g))\# S_{\bF[\bZ_2]}(g)][S(u^{-1})\# 1]\\
            &= [S_H(u)u^{-1}(g\cdot S_H(u)^{-1})]\# g = u^{-1}\# g = \phi(S(\rib^{-1})).
        \end{align*}
        From here, it is easy to see that that $\phi$ is a Hopf algebra isomorphism. Clearly, $(R^{(1)}\# 1)\otimes (R^{(2)}\# 1)=(\phi\otimes\phi)(R)$ is an $R$-matrix for $H\#_\sigma \bF[\bZ_2]$. The result follows.
    \end{proof}

    We now consider an example which shows that the decompositions given in Theorems \ref{ribbon-decomp} and \ref{sommer} are truly distinct.
    \begin{example}\label{decompositions-differ} 
        Consider the 4-dimensional Sweedler's Hopf algebra $H_4$ generated by $K, \xi$ subject to
        \begin{align*}
            K^2 &= 1, & \xi^2 &= 0, & K\xi&=-\xi K
        \end{align*}
        with Hopf algebra structure described by
        \begin{align*}    
            \Delta(K) &= K\otimes K, & \epsilon(K) &= 1, & S(K) &= K,\\
            \Delta(\xi) &= K\otimes \xi + \xi\otimes 1, & \epsilon(\xi) &= 0, & S(\xi) &= -K\xi.
        \end{align*}
        By \cite{panaite1999quasitriangular}, $H_4$ is ribbon with
        \begin{align*}
            R &= (1\otimes 1 + K\otimes 1 + 1\otimes K - K\otimes K)(1\otimes 1 + \xi\otimes K\xi)
        \end{align*}
        and $\mathbf{v} = 1$. $H_4$ is the first Nichols Hopf algebra $\mcK_1$, which are further discussed in Section \ref{DKn-section}.
        
        Theorem \ref{ribbon-decomp} implies $\tilde H_4\cong H_4\otimes \bF[\bZ_2]$. However, $S^2(\xi)=-\xi$, so, in Sommerh\"auser's decomposition $\tilde H_4\cong H_4\#_\sigma \bF[\bZ_2]$, the weak action of $\bF[\bZ_2]$ on $H_4$ is nontrivial (meanwhile, the cocycle is trivial in this case). In particular, $(\xi\# 1)(1\# g) = -(1\#g)(\xi\#1)$ in $H_4\#_\sigma \bF[\bZ_2]$. Thus, Sommerh\"auser's decomposition does not reduce to the tensor product decomposition of Theorem \ref{ribbon-decomp}.
    \end{example}
    
    \begin{corollary}\label{exact}
        Let $H$ be a finite-dimensional quasitriangular Hopf algebra. Then, the sequence $H\to \H\to \bF[\bZ_2]$ is a strictly exact sequence of quasitriangular Hopf algebras. In particular, there is an exact sequence of braided finite tensor categories: 
            $$\Vec_{\bZ_2}\to \Rep(\H)\to \Rep(H).$$
    \end{corollary}
    \begin{proof}
        By the correspondence between cleft extensions and cocycled crossed products (and as noted in \cite[Rmk. 3.14]{Andruskiewitsch2014hopftensor}), there is a cleft exact sequence $H\incl \H\twoheadrightarrow \bF[\bZ_2]$. This sequence is, in particular, is strictly exact. By \cite[Prop. 2.9]{BruguieresNatale11}, this gives rise to an exact sequence of tensor categories:
            $$\Rep(\bF[\bZ_2])\cong\Vec_{\bZ_2}\to \Rep(\H)\to \Rep(H).$$
        Since the maps $H\incl \H$ and $\H\twoheadrightarrow \bF[\bZ_2]$ preserve the $R$-matrix, the functors are braided.
    \end{proof}

    \section{Representations of $\H$}\label{rep-section}
    Throughout this section, $H$ is a finite-dimensional quasitriangular Hopf algebra over an algebraically closed field $\bF$ of characteristic 0, and $M$ is a finite-dimensional $H$-module. For an $\H$-module $N$, we denote the $H$-module obtained by restricting the action on $N$ to $H$ by $N|_H$.
    
    \subsection{Upgrading $H$-actions to $\H$-actions}\label{extension-props}
    By dominance, Corollary \ref{exact} shows that every $H$-module arises as an $H$-submodule of some $\H$-module (with its action restricted to $H$). Theorem \ref{action-extends} shows that, in fact, every $H$-module has an $\H$-action. Moreover, we obtain a characterization of simple and projective $\H$-modules in terms of simple and projective $H$-modules in Proposition \ref{simples-double} and Theorem \ref{indec-projectives-restrict}. Modules of this form are of particular interest in finite tensor categories, as the finiteness property can be characterized in terms of such objects \cite{etingof2004finite}.

    Subsection \ref{extension-props} is much more general than presented. Indeed, the coalgebra and antipode structures are not relevant here. All results hold identically for extensions of algebras $\mcA\subset \mcA[a]$ for which $a$ commutes with $\mcA$ and $a^2=b$ for some invertible $b\in \mcA$. This subsection may be further generalized to the case where $a^n=b$. As our primary interest is the formal ribbon extension, we phrase all results using the special case of $H\subset \H$, where $a = \rib$ and $b=uS(u)$.
    
    Lemma \ref{sqrt-uSu} is a technical result that we employ to show that $H$-modules can always be made into $\H$-modules and that there is a precise description in the indecomposable case. The proof presented here applies the holomorphic functional calculus to prove the result for $\bC$ and then completeness of the theory of algebraically closed fields of characteristic 0 to extend it to more general fields. The intersection in the applicability of these two results seems quite small.  Moreover, there is a more direct proof of Lemma \ref{sqrt-uSu} which relies on neither of these techniques, described in Remark \ref{std-proof}.
  
    \begin{lemma}\label{sqrt-uSu}
        Suppose $V$ is a finite-dimensional vector space over an algebraically closed field $\bF$ of characteristic 0 and that $B\in \Aut_{\bF}(V)$. Then, $B$ has a square root $A\in\Aut_{\bF}(V)$ with the following properties:
        \begin{enumerate}
            \item $A$ commutes with the centralizer of $B$ in $\End_{\bF}(V)$;
            \item if $\lambda$ is an eigenvalue of $A$, then $-\lambda$ is not an eigenvalue of $A$.
        \end{enumerate}
    \end{lemma}
    \begin{proof}
        We can formulate Lemma \ref{sqrt-uSu} in terms of logical sentences in the language of rings. For each $n\geq 1$, let $\phi_n$ be the sentence
        \begin{align*}
            \phi_n := ``&\forall b_{11},\dots, b_{nn}, (\det[b_{ij}]\neq 0\to\exists a_{11},\dots, a_{nn}, (\det[a_{ij}]\neq 0\wedge [a_{ij}]^2 = [b_{ij}] \\&\wedge\forall c_{11},\dots, c_{nn}, ([c_{ij}][b_{ij}]=[b_{ij}][c_{ij}]\to [c_{ij}][a_{ij}]=[a_{ij}][c_{ij}])\wedge \\
            &\wedge \forall \lambda (\det([a_{ij}] - \lambda I)=0\to \det([a_{ij}] + \lambda I)\neq 0)))".
        \end{align*}
        where $[a_{ij}]$ denotes the matrix whose entries are $a_{ij}, i,j=1,\dots, n$. That is, $\phi_n$ is the sentence ``for every $n\times n$ matrix $B$ with nonzero determinant, there is an $n\times n$ matrix $A$ which has nonzero determinant, squares to $B$, for every $n\times n$ matrix $C$, if $B$ commutes with $C$, then $A$ commutes with $C$, and, for any $\lambda$, if $\det(A-\lambda I)=0$, then $\det(A+\lambda I)\neq 0$.'' Observe that Lemma \ref{sqrt-uSu} is equivalent to $\phi_n$ being true for all $n$. By completeness of the theory of algebraically closed fields of characteristic 0, $\phi_n$ is true for $\bC$ if and only if $\phi_n$ is true for all algebraically closed fields of characteristic 0. Thus, it suffices to prove the result for $\bC$ only.

        Since $B$ has a finite spectrum $\sigma(B)$ and is invertible (so $0\notin \sigma(B)$), there is a holomorphic branch of the square root function $\sqrt{\cdot}:U\to\bC$ in a neighborhood $U$ of the spectrum of $B$. The holomorphic functional calculus gives an element $A=\sqrt{B}$ in the Banach algebra generated by $B$ such that $A^2 = B$. Moreover, since $V$ is finite-dimensional, the operator $A$ is a polynomial in $B$. In particular, $A\in Z_{\{B\}}(\End_\bF(V))$. Moreover, $A$ is invertible with $A^{-1} = AB^{-1}$. Finally, note that if $\delta\in \sigma(B)$, then $\sqrt{\delta}\in \sigma(A)$ (where this is the chosen branch of the square root) and $-\sqrt{\delta}\notin \sigma(A)$. All eigenvalues of $A$ arise this way. 
    \end{proof}
    
    It is worth noting that extending results obtained by (holomorphic, continuous, or Borel) functional calculi on $\bC$ to arbitrary algebraically closed fields of characteristic 0 using model theory is quite difficult in general. Sentences in the language of rings may only involve polynomials equations, greatly limiting the applicability of model theoretic techniques. However, general holomorphic functions need not send, for example, $\bar\mbbQ$ to itself. Thus, we should not expect that most results obtained by functional calculi to be applicable to arbitrary $\bF$ anyways.

    \begin{remark}\label{std-proof}
        Here is a more constructive proof of Lemma \ref{sqrt-uSu}. Let $b(x)=(x-\lambda_1)^{n_1}\dots(x-\lambda_k)^{n_k}$ be the minimal polynomial of $B$, where the $\lambda_i\in\bF^\times$ are distinct and nonzero. For each $i=1,\dots,n$, using the fact that $\bF$ is algebraically closed of characteristic 0, let $a_i(x)$ be the Taylor polynomial of (a branch of) $\sqrt{x}$ about $x=\lambda_i$ of degree $n_i-1$. Then, $a_i(x)$ satisfies $b_i(x)^2=x\pmod{(x-\lambda_i)^{n_i}}$. By the Chinese remainder theorem, there is a polynomial $a(x)$ such that $a(x)\equiv a_i(x)\pmod{(x-\lambda_i)^{n_i}}$ for all $i$. In particular, $a(x)$ satisfies $a(x)^2 = x\pmod{b(x)}$, so $a(x)^2 = x + p(x)b(x)$ for some polynomial $p(x)$. Therefore, if $A = a(B)$, then $A^2 = B + p(B)b(B) = B$. It follows that $A$ commutes with the centralizer of $B$ because it is polynomial in $B$. Moreover, $a(\sigma(B)) = \sigma(a(B))$, so each $\delta\in\sigma(B)$ may correspond to at most one $\lambda\in \sigma(a(B))$.
    \end{remark}

    \begin{theorem}\label{action-extends}
        Suppose $H$ is a finite-dimensional quasitriangular Hopf algebra and $M$ is a finite-dimensional $H$-module. Then, there is an $\H$-module $\tilde M$ such that, when the action is restricted to $H$, $\tilde M|_H\cong M$ as $H$-modules. Moreover, if $M$ is indecomposable, then there are precisely two such $\H$-modules, up to isomorphism.
    \end{theorem}
    \begin{proof}
        By Lemma \ref{sqrt-uSu}(1), there is a linear map $A\in \End_\bF(M)$ such that $A^2=uS(u)\cdot -$ and $A(h\cdot v) = h\cdot Av$ for all $v\in M$. As a vector space, set $\tilde M := M$ and define the action
        $$(a + b\rib)\cdot_{\tilde M} v := a\cdot_M v + b\cdot_M Av$$
        for $a, b\in H$ and $v\in \tilde M$. It is clear that this action does indeed define an $\H$-module and that $\tilde M|_H = M$. 
        
        Set $M^+=\tilde M$, and let $M^-$ be the $\H$-module with the conjugate action: 
        $$(a + b\rib)\cdot_{M^-} v := a\cdot_M v - b\cdot_M Av$$
        for any $a,b\in H$ and $m\in M$. Note that indeed $M^-|_H=M$. Note that $M^-\not\cong M^+$ as $\H$-modules, as the spectra of $\rib\cdot -$ differ by Lemma \ref{sqrt-uSu}(2). Consider $M\oplus \rib M$ as an $\H$-module in the obvious way. Define $F:M\oplus \rib M\to M^+\oplus M^-$ by
        $$m_1\oplus \rib m_2\mapsto (m_1+Am_2)\oplus (m_1-Am_2).$$
        It is easily shown that this map is an $\H$-linear isomorphism. Thus, if $M$ is indecomposable, the result follows by the uniqueness of the decomposition of $M\oplus \rib M$ into indecomposable $\H$-modules.
    \end{proof}

    \begin{corollary}
        Let $H$ be a finite-dimensional quasitriangular Hopf algebra and $M$ be a finite-dimensional indecomposable $H$-module. Then, $uS(u)\cdot -:M\to M$ has precisely one distinct eigenvalue.
    \end{corollary}

    \begin{remark}
        Note that $\rib\cdot -$ need not be a polynomial in $uS(u)\cdot -$ for general $\H$-modules. Let $H=\bC$ be the trivial quasitriangular complex Hopf algebra (with $R=1\otimes 1$). Consider the possible ways to make $\bC^2$ into a $\tilde\bC=\bC[\bZ_2]$-module. $\bC$ acts by scalar multiplication on $\bC^2$, so any choice of square root of $uS(u)\cdot - = \id_{\bC^2}$ suffices for the action of $\rib\cdot -$. For example, we may set $\rib\cdot - = \left(\begin{smallmatrix}
            1 & 0\\
            0 & -1
        \end{smallmatrix}\right)$, which is certainly not a polynomial in $\id_{\bC^2}$. In this case, $\rib\cdot -$ is diagonalizable with eigenvalues in $\{1, -1\}$. However, in the case of simple $H$-modules, $\rib\cdot -$ is just a scalar multiple of the identity.
    \end{remark}

    As in the proof of Theorem \ref{action-extends}, given an indecomposable $H$-module $M$, we will denote the two $\H$-modules which restrict to $M$ as $M^+$ and $M^-$. Note that the complex square root has no general notion of positive and negative square root. With the exception of the trivial $H$-module $V_1$, $M^+$ and $M^-$ can be interchanged throughout this work without loss of generality. For $V_1$, the $\H$-module $V^+_1$, where $\rib\cdot - = \id_{V_1^+}$, is the trivial $\H$-module, while the $\H$-module $V^-_1$, where $\rib\cdot - = -\id_{V_{1}^+}$, is a sort of sign module. In particular, the ribbon category tensor-generated by $V_1^-$ is $\Vecm$.

    \begin{proposition}\label{simples-double}
        Suppose $H$ is a finite-dimensional quasitriangular Hopf algebra and $N$ is a simple $\H$-module. Then, $N|_H$ is a simple $H$-module. 
    \end{proposition}
    \begin{proof}
        Suppose $M\subseteq N|_H$ is an $H$-submodule of $N$. Note that $\rib\cdot -:N\to N$ is an $\H$-linear map since $\rib\in Z(H)$, so there is some $c\in\bF^\times$ such that $\rib\cdot - = c\id_N$ by Schur's lemma. In particular, $\rib\cdot m = cm$ for all $m\in M\subseteq N$. Therefore, $M$ is an $\H$-submodule of $N$, so $M=N$ or $M=0$. Thus, $N|_H$ is a simple $H$-module.
    \end{proof}
    
    \begin{corollary}\label{semisimple-iff}
        Let $H$ be a finite-dimensional quasitriangular Hopf algebra. Then, $H$ is semisimple if and only if $\H$ is semisimple.
    \end{corollary}
    \begin{proof}
        Suppose $H$ is semisimple. Note that
        $$\dim\left(\bigoplus_{\substack{V\text{ simple }\\\text{$H$-module}}} \dim(V^+)V^+\oplus \dim(V^-)V^-\right) = 2\dim\left(\bigoplus_{\substack{V\text{ simple }\\\text{$H$-module}}} \dim(V)V\right) = 2\dim H = \dim \H.$$
        Thus, all simple $\H$-modules must be their own projective covers. The same equality of dimensions shows the opposite implication.
    \end{proof}

    \begin{theorem}\label{indec-projectives-restrict}
        Let $V$ be a finite-dimensional simple module over a finite-dimensional quasitriangular Hopf algebra $H$. Let $P$ be the projective cover of $V$ as $H$-modules and $P_\pm$ be the projective covers of $V^\pm$ as $\H$-modules. Then, $P_\pm|_H\cong P$ as $H$-modules.
    \end{theorem}
    \begin{proof}
        For a simple $H$-module $V$, let $P(V)$ denote its ($H$-module) projective cover. Note that
        $$\H = H\oplus H\rib\cong \bigoplus_{V\text{ simple}} (\dim V)(P(V)\oplus \rib P(V)).$$
        Thus, $P(V)\oplus P(V)\rib$ is projective for every $V$.

        Let $p:P\to V$ be an $H$-module covering map. Without loss of generality, $V^\pm|_H = V$ as $H$-modules. Define the maps $p_\pm:P\oplus \rib P\to V^{\pm}$ given by $p_\pm(h_1 + h_2\rib) = p(h_1)+\rib\cdot p(h_2)$ for $h_1, h_2\in H$. Note that these maps are clearly $\H$-linear. Since $p$ is nonzero, both are nonzero. That is, $P\oplus \rib P$ has nonzero maps into both $V^+$ and $V^-$. Given a simple module $W$, $P(W)$ is the only projective indecomposable module (up to isomorphism) with a nonzero map into $W$. Since $(P\oplus \rib P)|_H\cong P\oplus P$ as $H$-modules, the decomposition of $P\oplus \rib P$ into a direct sum of $\H$-modules may contain at most two indecomposable summands. It follows that $P\oplus \rib P\cong P_+\oplus P_-$ as $\H$-modules. The isomorphism can be considered as an $H$-module isomorphism so that
        $$P\oplus P\cong (P\oplus \rib P)|_H\cong P_+|_H\oplus P_-|_H$$
        as $H$-modules. The result follows by the uniqueness of decomposition of $P\oplus \rib P$ into indecomposable $H$-modules.
    \end{proof}
    
    \begin{corollary}\label{projectives-restrict}
        Suppose $H$ is a finite-dimensional quasitriangular Hopf algebra and $Q$ is a projective $\H$-module. Then, $Q|_H$ is a projective $H$-module. Moreover, every projective $H$-module arises this way.
    \end{corollary}

    The question of whether indecomposable $\H$-modules restrict to indecomposable $H$-module remains unclear to the author. However, Proposition \ref{simples-double} and Corollary \ref{ribbon-reps-factor} give a positive answer in the simplest cases. Thus, we conjecture that this is true in general. This would, in particular, imply that, if $P$ is a projective $H$-module, then any $\tilde P$ is also projective, which is not implied by Corollary \ref{projectives-restrict}.
    \begin{conj}
        Suppose $H$ is a finite-dimensional quasitriangular Hopf algebra and $N$ is a finite-dimensional indecomposable $\H$-module. Then, $N|_H$ is also indecomposable as an $H$-module.
    \end{conj}

\subsection{The tensor category $\Rep(\H)$}
    
    \begin{lemma}\label{tensors-restrict}
        Suppose $H$ is a finite-dimensional quasitriangular Hopf algebra and $N, N'$ are finite-dimensional $\H$-modules. Then, $(N\otimes N')|_H\cong N|_H\otimes N'|_H$ as $H$-modules.
    \end{lemma}

    For a braided category $\cC$, let $\cC'$ denote the M\"uger center of $\cC$, defined by
    $$\cC' = \{M\in \cC~|~\beta_{M',M}\circ\beta_{M,M'}=\id_{M\otimes M'} \text{ for all }M'\in \cC\},$$
    where $\beta:\otimes\to\otimes^{\mathrm{op}}$ is the braiding on $\cC$.
    \begin{theorem}
        Suppose $H$ is a finite-dimensional quasitriangular Hopf algebra. Then, the following are equivalent for an $H$-module $M$:
        \begin{enumerate}
            \item $M\in \Rep(H)'$,
            \item $\tilde M\in \Rep(\H)'$ for some $\tilde M$ such that $\tilde M|_H\cong M$,
            \item $\tilde M\in \Rep(\H)'$ for all $\tilde M$ such that $\tilde M|_H\cong M$.
        \end{enumerate}
    \end{theorem}
    \begin{proof}
        By Theorem \ref{action-extends}, every $H$-module has such an $\H$-module $\tilde M$, and the $R$-matrix acts identically on $M\otimes M'$ and $\tilde M\otimes \tilde M'$.
    \end{proof}
    
    For any finite-dimensional quasitriangular Hopf algebra $H$, $V_1^-$ is always in the M\"uger center of $\Rep(\tilde H)$. The ribbon category generated by $V_1^-$ is not $\sVec$. The category, which we have denoted $\Vecm$, has a nontrivial ribbon twist $\theta_{V^-}=-1$ on $V^-$ but a trivial braiding $\beta_{V^-, V^-} = \id_{V^-}$. In particular, $\Rep(\H)$ can never be a modular category nor be condensed to a modular category.  
    \begin{corollary}\label{muger-factorizable}
        If $H$ is a finite-dimensional factorizable Hopf algebra, then $\Rep(\H)'\cong \Vecm$ and is tensor generated by $V_1^-$.
    \end{corollary}
    \begin{proof}
        Suppose $N\in\Rep(\H)'$ is indecomposable. By factorizability of $H$, $N|_H\cong V_1^{\oplus n}$ as an $H$-module for some $n\geq 0$. In particular, $h\cdot v = \epsilon(h)v$ for all $h\in H$ and $v\in N$. The action $g\cdot -$ of the finite-order grouplike $g = u\rib^{-1}$ on $N$ must be diagonalizable. Therefore, $g\cdot -$ has an eigenspace decomposition so that $N=N^+\oplus N^-$ and $g\cdot v^\pm = \pm v^\pm$ for $v^\pm\in N^\pm$. Since $H$ acts via scalar multiples of the identity, any subspace of $N|_H$ is a direct summand of $N|_H$ as an $H$-module. By indecomposability of $N$, this implies that either $N^+$ is one-dimensional and $N^-=0$ or vice versa. These correspond to $N\cong V_1^+$ or $N\cong V_1^-$ respectively. 
    \end{proof}

        \cite[Rmk. 3.1]{eno2005fusion} introduces a notion of pivotalization $\tilde\cC$ of a fusion category $\cC$. It is later shown in Prop. 5.14 that the pivotal structure on $\tilde\cC$ is spherical. We show that $\Rep(\H)\cong \widetilde{\Rep(H)}$ as braided fusion categories, when these two constructions both apply. In this sense, the formal ribbon extension is a generalization of pivotalization to the non-semisimple braided case.

        The notion of pivotalization is well-defined for any finite tensor category $\cC$ with a monoidal natural isomorphism $\Phi:\Id_{\cC}\to (-)^{****}$. By Radford's formula for $S^4$, pivotalization is, in particular, possible for representation categories of finite-dimensional Hopf algebras (i.e., finite tensor categories with a fiber functor to $\Vec$) or finite-dimensional semisimple weak Hopf algebras (i.e., fusion categories). In general, however, such a $\Phi$ need not exist. For example, a general weak Hopf algebra does not have such a formula. There is a generalization of Radford's formula to finite tensor categories \cite{ENO2004Radford}, which leads to a weaker natural isomorphism $\Phi:\Id_{\cC}\to ((-)^{****})^N$ for some $N\geq 1$, which depends on the order of the distinguished invertible object. 
        \begin{definition}
            Let $\cC$ be a finite tensor category with a monoidal natural isomorphism $\Phi:\Id_\cC\to (-)^{****}$. The pivotalization of $(\cC, \Phi)$ is the category $\tilde\cC$ with objects $(M, \phi)$ where $M\in\cC$ and $\phi:M\to M^{**}$ satisfies $\phi^{**}\circ \phi = \Phi$ and morphisms 
            $$\Hom_{\tilde\cC}((M, \phi)\to (M',\psi)):=\{f\in\Hom_{\cC}(M\to M')| \psi\circ f = f^{**}\circ \phi\}.$$ 
            We give the pivotalization a tensor category structure by 
            $$(M, \phi)\otimes (M',\psi):=(M\otimes M', \phi\otimes \psi).$$
        \end{definition}
        For a fusion category or representation category of a finite-dimensional Hopf algebra, $\Phi$ is assumed to be the Radford isomorphism induced by Radford's formula for $S^4$. In which case, we speak of the pivotalization of $\cC$ rather than $(\cC, \Phi)$.
        
        \begin{theorem}\label{pivotalization-agrees}
            If $H$ is a finite-dimensional semisimple quasitriangular Hopf algebra, then $\Rep(\H)\cong \widetilde{\Rep(H)}$ as braided fusion categories.
        \end{theorem}
        \begin{proof}
            In this proof, we will use the fact that $\H = H\oplus H\rib^{-1}$.
            
            Given a finite-dimensional $\H$-module $N$, set $M = N|_H$. Let $\{f_i\}$ be a basis of $M^*$ and $\{f^i\}\subset M^{**}$ be the dual basis. For $x\in M$ and $g\in M^*$, define the map $\phi_N:M\to M^{**}$ by
            $$\phi_N(x)(g) := \sum_i f_i(R^{(2)}\rib^{-1} \cdot x)[S(R^{(1)})\cdot f^i](g) = \sum_i f_i(R^{(2)}\rib^{-1} \cdot x)f^i(g(S^3(R^{(1)})\cdot -)).$$

            Let $M$ be a finite-dimensional simple $H$-module and $\phi:M\to M^{**}$ an $H$-linear automorphism such that $\phi^{**}\circ\phi = d$. Let $\{e_i\}$ be a basis for $M$ and $\{e^i\}$ be the dual basis. Let $M^\phi = M$ as a vector space. For $m\in M$ and $a,b\in H$, we define the following explicit $\H$-action on $M^\phi$:
            $$(a+b\rib^{-1})\cdot m := a\cdot m + \sum_i [R^{(1)}\cdot \phi(m)](e^i)(bR^{(2)}\cdot e_i) = a\cdot m + \sum_i \phi(m)(e^i(S^2(R^{(1)})\cdot -))(bR^{(2)}\cdot e_i).$$
            Diagrammatically, the maps $\phi_N:N|_H\to (N|_H)^{**}$ and $\rib^{-1}\cdot -:M^\phi\to M^\phi$ may be written as follows:
            $$
                \phi_N = \begin{tikzpicture}[baseline=1cm]
                     \draw (0, -2) node[below] {$N|_H$} -- (0, -0.35) (-2.2, 0.35)  .. controls (-2.2, 1) and (0, 1) ..   (0, 1.35) -- (0, 3) node[above] {$(N|_H)^{**}$} (-3.4, 0.35) -- node[left] {$(N|_H)^*$} (-3.4, 1.35);
                     \draw[line width=2mm, white] (0, 0.35) .. controls (0, 1) and (-2.2, 1) ..   (-2.2, 1.35);
                     \draw (0, 0.35) node[above right] {$N|_H$} .. controls (0, 1) and (-2.2, 1) ..   (-2.2, 1.35) ;
                     \draw[rounded corners] (-1, -0.35) rectangle  node {$\rib^{-1}\cdot_N -$} (1, 0.35) (-3.9, -0.35) rectangle node {$\coev_{(N|_H)^*}$}  (-1.5, 0.35) (-3.9, 1.35) rectangle node {$\ev_{N|_H}$}  (-1.5, 2.05);
                \end{tikzpicture},\hspace{10mm}\rib^{-1}\cdot - =
                \begin{tikzpicture}[baseline=1cm]
                     \draw (0, -2) node[below] {$M^\phi$} -- (0, -0.35) (2, 0.35) .. controls (2, 1) and (0, 1) ..   (0, 1.35) -- (0, 3) node[above] {$M^\phi$} (3, 0.35) -- node[right] {$(M^\phi)^*$} (3, 1.35);
                     \draw[line width=2mm, white] (0, 0.35) .. controls (0, 1) and (2, 1) ..   (2, 1.35);
                     \draw (0, 0.35) node[above left] {$(M^\phi)^{**}$} .. controls (0, 1) and (2, 1) ..   (2, 1.35);
                     \draw[rounded corners] (0.5, -0.35) rectangle  node {$\phi$} (-0.5, 0.35) (3.5, -0.35) rectangle node {$\coev_{M^\phi}$}  (1.5, 0.35) (3.5, 1.35) rectangle node {$\ev_{(M^\phi)^*}$}  (1.5, 2.05);
                \end{tikzpicture}.
            $$
            This construction is well-known (see, for instance, \cite{Henriques2015CategorifiedTF}). As a composition of $H$-linear maps, $\phi_N$ and $\rib^{-1}\cdot -$ are $H$-linear. These constructions are inverse to each other in the sense that $\phi_{M^{\phi}} = \phi$ and $(N|_H)^{\phi_N} = N$. 

            Consider the following functors $F:\Rep(\H)\to \widetilde{\Rep(H)}$ and $G:\widetilde{\Rep(H)}\to \Rep(\H)$: for an $H$-module $M$ and $\H$-module $N$, we define
            \begin{align*}
                F(N) &= (N|_H, \phi_{N}),\\
                G((M, \phi)) &= M^\phi,
            \end{align*}
            The functors are defined on morphisms in the obvious way. Moreover, we have $G(F(N)) = N$ and $F(G((M, \phi))) = (M, \phi)$. Thus, $F$ and $G$ form a linear isomorphism of categories. By construction, $F$ and $G$ are indeed braided tensor functors. Thus, the result follows.
        \end{proof}

\section{The formal ribbon extension of doubled Nichols Hopf algebras}\label{DKn-section}

    \subsection{Doubled Nichols Hopf algebras $D\Kn$}
        Given how nicely $\H$-modules can be described in relation to $H$-modules, one might wonder if anything nontrivial is being studied here. In particular, is it always true that the fusion rules of $\Rep(\H)$ and $\Rep(H)\boxtimes\Vecm$ agree? Proposition \ref{not-deligne} provides a disproof using the doubled Nichols Hopf algebra $D\Kn$ for odd $n$, even for the full subcategories generated by just simple and projective $\RKn$-modules. We recall the study of Nichols Hopf algebras $\Kn$ and their doubles $D\Kn$ from \cite{modulardata2024}.
        
        Let $n$ be a positive integer. The Nichols Hopf algebra $\Kn$ is the $2^{n+1}$-dimensional complex unital algebra which has generators $K, \xi_1,\dots,\xi_n$ subject to the following relations: for all $i,j=1,\dots, n$,
        \begin{align*}
            K^2 &= 1, & \xi_i^2 &= 0, & \xi_i\xi_j &= -\xi_j\xi_i, & K\xi_i &= -\xi_i K.
        \end{align*}
        In particular, $\Kn$ is a crossed product $\bigwedge^* \bC^n\rtimes \bC[\bZ_2]$ of an exterior algebra on an $n$-dimensional space by the $\bZ_2$-group algebra. The algebra $\Kn$ has a natural Hopf algebra structure with the following operations: for all $i=1,\dots, n$,
        \begin{align*}
            \Delta(K) &= K\otimes K,&\epsilon(K) &= 1,&S(K)&=K,\\
            \Delta(\xi_i) &= K\otimes \xi_i + \xi_i\otimes 1,&\epsilon(\xi_i) &= 0,&S(\xi_i)&=-K\xi_i.
        \end{align*}
        Let $W = \{\xi_1^{a_1}\xi_2^{a_2}\dots\xi_n^{a_n} | a_i\in\{0,1\}\}$. Then, $\Kn$ has a natural basis given by $W\cup KW$. Set $\mathbf{I}(w) = \{i | a_i=1\}$ for $w=\xi_1^{a_1}\xi_2^{a_2}\dots\xi_n^{a_n}\in W$ and $|w| = |\mathbf{I}(w)| = \sum_{i=1}^n a_i$. By \cite[Prop. 11]{panaite1999quasitriangular}, there is a large family of $R$-matrices for $\Kn$ and each has a unique compatible ribbon element.

        The Drinfeld double $D\Kn$ of $\Kn$ is a $2^{2n+2}$-dimensional complex Hopf algebra. $D\Kn$ is generated by $K, \bar K, \xi_1,\dots,\xi_n, \bar\xi_1,\dots, \bar\xi_n$, where the sets $\{K, \xi_1, \dots, \xi_n\}$ and $\{\bar K, \bar\xi_1, \dots, \bar\xi_n\}$ each generate a copy of the Hopf algebra $\Kn$ with the operations and relations described above. Moreover, they have the following additional relations: for $i,j=1,\dots, n$,
        \begin{align*}
            K\bar K &= \bar KK, & K\bar\xi_i &= -\bar\xi_i K, & \bar K\xi_i &= -\xi_i\bar K, & \xi_i\bar\xi_j &= \delta_{i=j}(1 - K\bar K) - \bar\xi_j\xi_i.
        \end{align*}
        If we set $\overline{\xi_1^{a_1}\xi_2^{a_2}\dots\xi_n^{a_n}} = \bar\xi_1^{a_1}\bar\xi_2^{a_2}\dots\bar\xi_n^{a_n}$ for $\xi_1^{a_1}\xi_2^{a_2}\dots\xi_n^{a_n}\in W$ and set $\bar W = \{\bar w|w\in W\}$, then $W\bar W\cup KW\bar W\cup \bar KW\bar W\cup K\bar KW\bar W$ is a basis of $D\Kn$. 

        By the construction of Drinfeld doubles, $D\Kn$ has a natural $R$-matrix which makes the algebra factorizable:
        $$R = \sum_{w\in W} \frac{(-1)^{\lfloor|w|/2\rfloor}}{2}(w\otimes \bar w(1+\bar K) + Kw\otimes \bar w(1-\bar K)).$$
        When $n$ is even, $D\Kn$ has two compatible ribbon elements ${\mathbf{v}}_0$ and ${\mathbf{v}}_1$ given by
        $${\mathbf{v}}_i = \sum_{w\in W} \frac{(-1)^{\lfloor|w|/2\rfloor}}{2}((-1)^{|w|+i}(K-\bar K)+1+K\bar K)w\bar w.$$

    \subsection{A presentation of $\RKn$} 
    \begin{theorem}
        $D\Kn$ with its canonical $R$-matrix is not ribbon for odd $n$.
    \end{theorem}
    \begin{proof}
    	Note that, in $\Kn$, $S^{4} = \id_{\Kn}$. When $n$ is odd, the distinguished grouplikes in $\Kn$ and $\Kn^*$ are $K$ and $\bar K$ respectively. By \cite[Thm. 3]{kauffman1993necessary}, it suffices to show that $K\in\Kn$ has no grouplike square-root. We show that $1$ and $K$ are the only grouplikes in $\Kn$. Suppose $x = \sum_{w\in W} (c_{1}(w)1 + c_{K}(w)K)w\in \Kn$ is a grouplike. Then, the following two equalities hold: 
    	\begin{align*}
    		\Delta(x) &= \sum_{w_{1},w_{2}\in W} (c_{1}(w_{1})1 + c_{K}(w_{1})K)w_{1}\otimes(c_{1}(w_{2})1 + c_{K}(w_{2})K)w_{2},\\
    		&= \sum_{w\in W}(c_{1}(w)1\otimes 1 + c_{K}(w)K\otimes K)\prod_{i\in \mathbf{I}(w)} (K\otimes \xi_i + \xi_i\otimes 1).
    	\end{align*}
        Note that in the first equality, for any $w\in W$, there is a term of the form $c_1(w)^2 w\otimes w$. However, in the second equality, every summand of the form $w_1\otimes w_2$ necessarily satisfies $\mathbf{I}(w_1)\cap \mathbf{I}(w_2)=\varnothing$. Thus, $c_1(w) = 0$ for any $w\neq 1$. Similarly, $c_K(w)=0$ for any $w\neq 1$. It is straightforward from here to see that only $1$ and $K$ are grouplikes for $\Kn$. In particular, neither is a square root of $K$.
    \end{proof}

    The following lemma provides a formula for converting between sums of the form $\sum_{w\in W} h_w\bar w w$ where $h_w$ depends only on the length of $w$ to sums of the form $\sum_{w\in W} k_w w\bar w$. Sums of this form are extremely common in our calculations. The lemma is symmetric, in that swapping all $w$'s with $\bar w$'s (and vice versa) does not change the validity of the formulas.
    
    \begin{lemma}\label{sum-by-length}
        Let $f:\{0, 1,\dots, n\}\to D\Kn$ be any function. Then,
        \begin{align*}
            \sum_{w\in W} f(|w|)\bar w w&=\sum_{w\in W} \left((-1)^{|w|}f(|w|) + \right.\\&\left.+ (-1)^{\lfloor (|w|+1)/2\rfloor}\sum_{\ell=1}^{n-|w|} \binom{n-|w|}{\ell}(-1)^{\lfloor (\ell+|w|)/2\rfloor}2^{\ell-1}f(\ell+|w|)(1 - K\bar K)\right)w\bar w.
        \end{align*}
        In particular,
        \begin{align*}
            \sum_{w\in W} f(|w|)(1+K\bar K)\bar w w&=\sum_{w\in W} (-1)^{|w|}f(|w|)(1+K\bar K)w\bar w,\\
            \sum_{w\in W} f(|w|)(1-K\bar K)\bar w w&=\sum_{w\in W}(-1)^{\lfloor (|w|+1)/2\rfloor}\sum_{\ell=0}^{n-|w|} \binom{n-|w|}{\ell}(-1)^{\lfloor (\ell+|w|)/2\rfloor}2^{\ell}f(\ell+|w|)(1 - K\bar K)w\bar w.
        \end{align*}
    \end{lemma}
    Note: the index of $\ell$ starts at 1 in the general case, but 0 in the case when the sum has a $(1-K\bar K)$-coefficient.
    \begin{proof}
        First, we note the following formula for any $w\in W$:
        $$\bar ww = (-1)^{\lfloor |w|/2\rfloor}\sum_{\mathbf{I}(w')\subseteq \mathbf{I}(w)} (-1)^{\lfloor (|w'|+1)/2\rfloor}(1 - K\bar K)^{|w|-|w'|} w'\bar w'.$$
        Then,
        \begin{align*}
            \sum_{w\in W} f(|w|)\bar w w &= \sum_{w\in W} f(|w|)(-1)^{\lfloor |w|/2\rfloor}\sum_{\mathbf{I}(w')\subseteq \mathbf{I}(w)} (-1)^{\lfloor (|w'|+1)/2\rfloor}(1 - K\bar K)^{|w|-|w'|} w'\bar w'
            \intertext{Note that there are $\binom{n-|w'|}{\ell}$ choices of $w\in W$ so that $\mathbf{I}(w')\subsetneq \mathbf{I}(w)$ so that $\ell+|w'| = |w|$. Thus, by indexing the sum by $\ell$ and relabeling $w'$ by $w$, we obtain the following:}
            &= \sum_{w\in W} (-1)^{\lfloor (|w|+1)/2\rfloor}\sum_{\ell=0}^{n-|w|} \binom{n-|w|}{\ell}(-1)^{\lfloor (\ell+|w|)/2\rfloor}f(\ell+|w|)(1 - K\bar K)^{\ell} w\bar w.
        \end{align*}
        Finally, if $\ell > 0$, then $(1 - K\bar K)^{\ell} = 2^{\ell-1}(1 - K\bar K)$, so the sum can be broken into the two sums given in the lemma statement. The two special cases follow directly.
    \end{proof}

    \begin{lemma}\label{Drinfeld-calculations}
    Let $n$ be a positive integer. Then, the Drinfeld element $u\in D\Kn$ satisfies
    \begin{align*}
        u &= \sum_{w\in W} \frac{(-1)^{\lfloor|w|/2\rfloor}}{2}((-1)^{n+|w|}(1-K\bar K)+K+\bar K)w\bar w,\\
        u^{-1} &= \sum_{w\in W} \frac{(-1)^{\lfloor(|w|+1)/2\rfloor}}{2}((-1)^{n}(1-K\bar K)+K+\bar K)w\bar w,\\
        S(u) &= \sum_{w\in W} \frac{(-1)^{\lfloor|w|/2\rfloor}}{2}((-1)^{|w|}(1-K\bar K)+K+\bar K)w\bar w,\\
        uS(u) &= \frac{(-1)^{n}}{2}(1 - K\bar K) + (1+K\bar K)\sum_{w\in W}(-1)^{\lfloor|w|/2\rfloor}2^{|w|-1}w\bar w,\\
        uS(u)^{-1} &= (K\bar K)^n.
    \end{align*}
    \end{lemma}
    We see that $S(u)=u$ only when $n$ is even, wherein $uK$ and $u\bar K$ are both ribbon elements in $D\Kn$. When $n$ is odd, $K$ and $\bar K$ still implement the antipode, but it is no longer the case that $uK$ and $u\bar K$ are ribbon elements since $\bar K^2\neq uS(u)^{-1}\neq K^2$.
    \begin{proof}
        By definition,
        \begin{align}
            u &= \sum_{w\in W} \frac{(-1)^{\lfloor|w|/2\rfloor}}{2}( (1+\bar K)(-\bar K)^{|w|}\bar w w + (1-\bar K)(-\bar K)^{|w|}\bar w Kw)\nonumber\\ 
            &= \sum_{w\in W} \frac{(-1)^{\lfloor(|w|+1)/2\rfloor}}{2}(1+K+\bar K-K\bar K)\bar w w.\label{u-barww}
        \end{align}
        By applying Lemma \ref{sum-by-length} to the functions $f(k) = \frac{(-1)^{\lfloor (k+1)/2\rfloor}}{2}K(1+K\bar K)$ and $f(k) = \frac{(-1)^{\lfloor (k+1)/2\rfloor}}{2}(1 - K\bar K)$, we get, by the binomial theorem,
        \begin{align*}
            &=\sum_{w\in W} \frac{(-1)^{\lfloor|w|/2\rfloor}}{2}(K+\bar K)w\bar w +\sum_{w\in W}(-1)^{\lfloor (|w|+1)/2\rfloor}\sum_{\ell=0}^{n-|w|} \binom{n-|w|}{\ell}(-1)^{\ell+|w|}2^{\ell-1}(1 - K\bar K)w\bar w,\\
            &= \sum_{w\in W} \frac{(-1)^{\lfloor|w|/2\rfloor}}{2}((-1)^{n+|w|}(1-K\bar K) + K + \bar K)w\bar w.
        \end{align*}
        By Proposition \ref{Drinfeld-props}, we see that
        \begin{align*}
            u^{-1} &= \sum_{w\in W} \frac{(-1)^{\lfloor|w|/2\rfloor}}{2}(\bar w(1+\bar K)S^2(w) + \bar w (1-\bar K)S^2(Kw)),\\
             &= \sum_{w\in W} \frac{(-1)^{\lfloor|w|/2\rfloor}}{2}(((-1)^{|w|}(1-K\bar K)+K+\bar K)\bar ww).
             \intertext{Again, we apply the two special cases of Lemma \ref{sum-by-length} to get}
             &= \sum_{w\in W} \frac{(-1)^{\lfloor (|w|+1)/2\rfloor}}{2}\left((K+\bar K) + \sum_{\ell=0}^{n-|w|} \binom{n-|w|}{\ell}(-1)^{\ell + |w|}2^\ell (1-K\bar K))\right)w\bar w,\\
             &= \sum_{w\in W} \frac{(-1)^{\lfloor (|w|+1)/2\rfloor}}{2}\left((K+\bar K) + (-1)^{n} (1-K\bar K))\right)w\bar w.
        \end{align*}
        Now, we compute $S(u)$ using Equation \eqref{u-barww}:
        \begin{align*}
            S(u) &= \sum_{w\in W} \frac{(-1)^{\lfloor(|w|+1)/2\rfloor}}{2}(1+K+\bar K-K\bar K)(-K)^{|w|}w(-\bar K)^{|w|}\bar w,\\
            &= \sum_{w\in W} \frac{(-1)^{\lfloor|w|/2\rfloor}}{2}((-1)^{|w|}(1-K\bar K)+K+\bar K)w\bar w.
        \end{align*}
        Next, we compute the product $uS(u)$. First, let $\hat w = \prod_{i\in\mathbf{I}(w)} \xi_i\bar\xi_i$. Note that $\xi_i\bar\xi_i\xi_j\bar\xi_j = \xi_j\bar\xi_j\xi_i\bar\xi_i$ for $i\neq j$, so this product does not need a specified ordering. Then, we can write $u$ and $S(u)$ as
        \begin{align*}
            u &= \sum_{w\in W} \frac{1}{2}((-1)^{n+|w|}(1-K\bar K)+K+\bar K)\hat w,\\
            S(u) &= \sum_{w\in W} \frac{1}{2}((-1)^{|w|}(1-K\bar K)+K+\bar K)\hat w.
        \end{align*}
        We compute $uS(u)$ as follows:
        \begin{align*}
            uS(u) &= \sum_{w_1,w_2\in W} \frac{1}{4}((-1)^{|w_1|+n}(1-K\bar K)+K+\bar K)\hat w_1((-1)^{|w_2|}(1-K\bar K)+K+\bar K)\hat w_2,\\
            &= \sum_{w_1,w_2\in W} \frac{1}{2}((-1)^{|w_1|+|w_2|+n}(1 - K\bar K) + 1 + K\bar K)\prod_{i\in \mathbf{I}(w_1)\cap \mathbf{I}(w_2)}\xi_i\bar \xi_i\xi_i\bar \xi_i\prod_{i\in \mathbf{I}(w_1)\vartriangle \mathbf{I}(w_2)} \xi_i\bar \xi_i,\\
            &= \sum_{w_1,w_2\in W} \frac{1}{2}((-1)^{|w_1|+|w_2|+n}(1 - K\bar K) + 1 + K\bar K)(1 - K\bar K)^{|\mathbf{I}(w_1)\cap \mathbf{I}(w_2)|}\hat{\mathbf{w}}(\mathbf{I}(w_1)\cup \mathbf{I}(w_2)),
        \end{align*}
        where $\vartriangle$ denotes symmetric difference and, if $I\subseteq \{1,\dots, n\}$, then $\hat{\mathbf{w}}(I)=(\xi_1\bar\xi_1)^{a_1}\dots(\xi_n\bar\xi_n)^{a_n}$, where $a_i = 1$ if $i\in I$ and $a_i = 0$ otherwise. We can reindex the sum to take advantage of the $(1-K\bar K)$ multiplier as follows:
        \begin{align*}
            &= \sum_{\mathbf{I}(w_1)\cap \mathbf{I}(w_2)=\varnothing} \frac{1}{2}(K + \bar K)\hat w_1\hat w_2 + \sum_{w_1, w_2\in W} 2^{|\mathbf{I}(w_1)\cap \mathbf{I}(w_2)|-1}(-1)^{|w_1|+|w_2|+n}(1 - K\bar K)\hat{\mathbf{w}}(\mathbf{I}(w_1)\cup \mathbf{I}(w_2)).
            \intertext{Given $w\in W$, consider all summands of the form $(1-K\bar K)\hat w$. For each $\ell_1=0,\dots,|w|$, there are $\binom{|w|}{\ell_1}$ different $w_1\in W$ such that $\mathbf{I}(w_1)\subseteq \mathbf{I}(w)$ and $|w_1|=\ell_1$. Fix such a $w_1$. Then, for each $\ell_2=|w|-\ell_1,\dots,|w|$, there are $\binom{\ell_1}{\ell_1+\ell_2-|w|}$ different $w_2\in W$ with $|w_2|=\ell_2$ and $\hat{\mathbf{w}}(\mathbf{I}(w_1)\cup \mathbf{I}(w_2)) = \hat w$. In this case $|\mathbf{I}(w_1)\cap \mathbf{I}(w_2)|=\ell_1+\ell_2-|w|$. Next, we consider all summands of the form $(1+K\bar K)\hat w$. For fixed $w_1\subseteq \mathbf{I}(w)$, there is exactly one $w_2$ such that $(K+\bar K)\hat w_1\hat w_2=\hat w$. There are $2^{|w|}$ such $w_1$. Moreover, the summand $\frac{1}{2}(1 + K\bar K)\hat w_1\hat w_2$ depends only on $w$. Indexing by $m=|w|-\ell_2$, the above sum can be rewritten as}
            &= \sum_{w\in W}2^{|w|-1}(1 + K\bar K)\hat w + \sum_{w\in W}\sum_{\ell_1=0}^{|w|}\sum_{m=0}^{\ell_1}\binom{|w|}{\ell_1}\binom{\ell_1}{m}2^{\ell_1-m-1}(-1)^{\ell_1+|w|-m+n}(1 - K\bar K)\hat{w}\\
            &= \sum_{w\in W}2^{|w|-1}(1 + K\bar K)\hat w + \sum_{w\in W}\sum_{\ell_1=0}^{|w|}\frac{\binom{|w|}{\ell_1}(-1)^{\ell_1+|w|+n}}{2}(1 - K\bar K)\hat{w}\\
            &= \sum_{w\in W}2^{|w|-1}(1 + K\bar K)\hat w + \frac{(-1)^{n}}{2}(1-K\bar K).
        \end{align*}
        The equation for $uS(u)$ follows. Finally, note that 
        $$(K\bar K)^n ((-1)^\ell(1-K\bar K) + K + \bar K) = (-1)^{n+\ell}(1-K\bar K)+K+\bar K,$$ 
        so $u = (K\bar K)^n S(u)$, and the last equation follows.
    \end{proof}
    
    By Theorem \ref{sommer}, for odd $n$, the ribbon extension $\RKn$ is obtained by formally adding a grouplike $\tilde k$ to $D\Kn$ such that $\tilde k^2 = K\bar K$ and $S^2(x) = \tilde kx\tilde k^{-1}$. The latter is equivalent to requiring that $\tilde k$ commutes with $K, \bar K$ and anticommutes with each $\xi_i, \bar\xi_i$. Note that since $\bar K = K\tilde k^2$, it need not be included in the generators. Thus, we can provide a simple presentation of $\RKn$. 
    \begin{corollary}
        Let $n$ be odd. The Hopf algebra $\RKn$ is generated (as an algebra) by the elements $K, \tilde k, \xi_1, \dots, \xi_n, \bar\xi_1,\dots,\bar\xi_n$ which satisfy the following relations: for $i\neq j$,
        \begin{align*}
            K^2 &= 1, & \xi_i^2 &= 0, & \xi_i\xi_j &= -\xi_j\xi_i, & K\xi_i &= -\xi_i K,\\
            \tilde k^4 &= 1, & \bar\xi_i^2 &= 0, & \bar\xi_i\bar\xi_j &= -\bar\xi_j\bar\xi_i, & \tilde k\bar\xi_i &= -\bar\xi_i \tilde k,\\
            K\tilde k &= \tilde k K, & \bar\xi_i\xi_i &= 1 - \tilde k^2 - \xi_i\bar\xi_i, & \xi_i\bar\xi_j &= -\bar\xi_j\xi_i, & K\bar\xi_i&=-\bar\xi_i K, & \tilde k\xi_i &= -\xi_i \tilde k.
        \end{align*}
    \end{corollary}

\subsection{Representations of $\RKn$}
    For even $n$, by Corollary \ref{ribbon-reps-factor}, $\Rep(\RKn)\cong \Rep(D\Kn)\boxtimes \Vecm$. For odd $n$, the behavior is quite similar, though the category does not factor so nicely. Throughout this section, we assume $n$ is odd.

    In \cite{modulardata2024}, it is shown that there are four simple $D\Kn$-modules, two of which are their own projective covers. There are two one-dimensional simple $D\Kn$-modules $V_1$ and $V_{K\bar K}$. The projective covers $P_1$ and $P_{K\bar K}$ of these modules are each of dimension $2^{2n}$. There are two $2^n$-dimensional projective simple $D\Kn$-modules $V_K$ and $V_{\bar K}$. It follows from Theorems \ref{simples-double} and \ref{indec-projectives-restrict} that all of these counts double. There are eight simple $\RKn$-modules. There are four one-dimensional simple $\RKn$-modules $V_1^\pm$ and $V_{K\bar K}^\pm$. The projective covers $P_1^\pm$ and $P_{K\bar K}^\pm$ of these modules are each of dimension $2^{2n}$. There are four $2^n$-dimensional projective simple $\RKn$-modules $V_K^\pm$ and $V_{\bar K}^\pm$. We now give explicit descriptions of these modules in terms of our presentation of $\RKn$.
    
    Aligning with the notation established in \cite{modulardata2024}, the one-dimensional simple $\RKn$-modules $V_{1}^+, V_{1}^-, V_{K\bar K}^+, V_{K\bar K}^-$ are described by 
    \begin{alignat*}{5}
        \forall \lambda\in V_{1}^+&:& \hspace{1cm}K\cdot\lambda &=& \lambda, \hspace{1cm}\tilde k\cdot \lambda &=& \lambda, \hspace{1cm}\xi_i\cdot\lambda=\bar\xi_i\cdot\lambda &=&\, 0,\\
        \forall \lambda\in V_{1}^-&:&\hspace{1cm}K\cdot\lambda &=& \lambda, \hspace{1cm}\tilde k\cdot \lambda &=&\, -\lambda, \hspace{1cm}\xi_i\cdot\lambda=\bar\xi_i\cdot\lambda &=&\, 0,\\
        \forall \lambda\in V_{K\bar K}^+&:&\hspace{1cm}K\cdot\lambda &=&\, -\lambda, \hspace{1cm}\tilde k\cdot \lambda &=& \lambda, \hspace{1cm}\xi_i\cdot\lambda=\bar\xi_i\cdot\lambda &=&\, 0,\\
        \forall \lambda\in V_{K\bar K}^-&:&\hspace{1cm}K\cdot\lambda &=&\, -\lambda, \hspace{1cm}\tilde k\cdot \lambda &=&\, -\lambda,\hspace{1cm} \xi_i\cdot\lambda=\bar\xi_i\cdot\lambda &=&\, 0.
    \end{alignat*} 
    Observe that $V_1^{\pm}|_{D\Kn}=V_1$, $V_{K\bar K}^{\pm}|_{D\Kn}=V_{K\bar K}$.
    The projective covers $P_1^{\pm}$ and $P_{K\bar K}^\pm$ are given by
    \begin{align*}
        P_1^+ &= \RKn(1 + K)(1 + \tilde k + \tilde k^2 + \tilde k^3),\\
        P_1^- &= \RKn(1 + K)(1 - \tilde k + \tilde k^2 - \tilde k^3),\\
        P_{K\bar K}^+ &= \RKn(1 - K)(1 + \tilde k + \tilde k^2 + \tilde k^3),\\
        P_{K\bar K}^- &= \RKn(1 - K)(1 - \tilde k + \tilde k^2 - \tilde k^3).
    \end{align*}
    They are indeed projective since $P_1^+ \oplus P_1^-\oplus P_{K\bar K}^+\oplus P_{K\bar K}^-\oplus \RKn(1 - \tilde k^2) = \RKn$. The covering maps $P\to V$ are all generated by $(1\pm K)(1\pm \tilde k + \tilde k^2 \pm \tilde k^3)\mapsto 1$. The dimensions of the $P_{1}^\pm$ and $P_{K\bar K}^\pm$ are $2^{2n}$.

    Set, as vector spaces, $V_{K}^{\pm}=V_{\bar K}^{\pm}=(\bC^2)^{\otimes n}$. Define the matrices $\Xi = \begin{bmatrix}
        0 & \sqrt{2}\\
        0 & 0
    \end{bmatrix}$ and $\sigma_Z = \begin{bmatrix}
        1 & 0\\
        0 & -1
    \end{bmatrix}$. Set $\Xi_j = \sigma_Z^{\otimes j-1}\otimes \Xi\otimes I_2^{\otimes n-j}$ and $\bar\Xi_i = \sigma_Z^{\otimes j-1}\otimes \Xi^T\otimes I_2^{\otimes n-j}$. Then, we define the following $\RKn$-module actions: for $j=1,\dots, n$,
    \begin{alignat*}{6}
        \forall v\in V_{K}^+&:& \hspace{1cm}K\cdot v &=& \sigma_Z^{\otimes n} v, \hspace{1cm}\tilde k\cdot v &=& i\sigma_Z^{\otimes n} v, \hspace{1cm}\xi_j\cdot v &=&\, \Xi_j v, \hspace{1cm}\bar\xi_j\cdot v &=&\, \bar\Xi_j v,\\
        \forall v\in V_{K}^-&:& \hspace{1cm}K\cdot v &=& \sigma_Z^{\otimes n} v, \hspace{1cm}\tilde k\cdot v &=& -i\sigma_Z^{\otimes n} v, \hspace{1cm}\xi_j\cdot v &=&\, \Xi_j v, \hspace{1cm}\bar\xi_j\cdot v &=&\, \bar\Xi_j v,\\
        \forall v\in V_{\bar K}^+&:& \hspace{1cm}K\cdot v &=& -\sigma_Z^{\otimes n} v, \hspace{1cm}\tilde k\cdot v &=& i\sigma_Z^{\otimes n} v, \hspace{1cm}\xi_j\cdot v &=&\, \Xi_j v, \hspace{1cm}\bar\xi_j\cdot v &=&\, \bar\Xi_j v,\\
        \forall v\in V_{\bar K}^-&:& \hspace{1cm}K\cdot v &=& -\sigma_Z^{\otimes n} v, \hspace{1cm}\tilde k\cdot v &=& -i\sigma_Z^{\otimes n} v, \hspace{1cm}\xi_j\cdot v &=&\, \Xi_j v, \hspace{1cm}\bar\xi_j\cdot v &=&\, \bar\Xi_j v.
    \end{alignat*}
    Observe that again $V_K^{\pm}|_{D\Kn}=V_K$ and $V_{\bar K}^{\pm}|_{D\Kn}=V_{\bar K}$. Thus, these modules are necessarily simple and are, in particular, singly generated by $|0\rangle^{\otimes n}$. Let $f:V_K^+\to V_K^-$ be a $\RKn$-linear map. By considering the action of $\frac{\xi_1\bar\xi_1}{2}\dots\frac{\xi_n\bar\xi_n}{2}$, we see that $f(|0\rangle^{\otimes n}) = c|0\rangle^{\otimes n}$ for some $c\in \bC$. Therefore, 
    $$-if(|0\rangle^{\otimes n}) = \tilde k\cdot f(|0\rangle^{\otimes n}) = f(\tilde k\cdot|0\rangle^{\otimes n}) = i f(|0\rangle^{\otimes n}).$$
    It follows that $f$ is the zero map, so $V_K^+$ and $V_K^-$ must be distinct. The same argument distinguishes $V_{\bar K}^+$ and $V_{\bar K}^-$.
    
    \begin{proposition}\label{jordan-decomp}
        The Jordan decomposition of $P^{+}_{1}$ and $P^-_{K\bar K}$ consists of $2^{2n-1}$ copies of $V_1^+$ and $2^{2n-1}$ copies of $V_{K\bar K}^-$. The Jordan decomposition of   $P^{-}_{1}$ and $P^+_{K\bar K}$ consists of $2^{2n-1}$ copies of $V_1^-$ and $2^{2n-1}$ copies of $V_{K\bar K}^+$. In particular, the Cartan matrix is
        $$\begin{pmatrix}
            2^{2n-1} & 0 & 0 & 2^{2n-1} & 0 & 0 & 0 & 0\\
            0 & 2^{2n-1} & 2^{2n-1} & 0 & 0 & 0 & 0 & 0\\
            0 & 2^{2n-1} & 2^{2n-1} & 0 & 0 & 0 & 0 & 0\\
            2^{2n-1} & 0 & 0 & 2^{2n-1} & 0 & 0 & 0 & 0\\
            0 & 0 & 0 & 0 & 1 & 0 & 0 & 0\\
            0 & 0 & 0 & 0 & 0 & 1 & 0 & 0\\
            0 & 0 & 0 & 0 & 0 & 0 & 1 & 0\\
            0 & 0 & 0 & 0 & 0 & 0 & 0 & 1
        \end{pmatrix}.$$
    \end{proposition}
    \begin{proof}
        This follows from the proof of \cite[Thm. 7.2.4]{modulardata2024}, noting that $\tilde k$ anticommutes with $\xi_i$ for all $i$.
    \end{proof}

    As expected, $uS(u)$ acts as $\id_M$ on the $D\Kn$-modules $M=V_1, V_{K\bar K}$ and $-\id_M$ on $M=V_K, V_{\bar K}$. However, $uS(u)$ does not act as a scalar multiple of $\id_M$ on $M=P_1$ or $M=P_{K\bar K}$.

    \begin{proposition}\label{fusion-rules}
        Let $\cdot:\{+, -\}^2\to\{+, -\}$ be given by $\pm\cdot\pm = +$ and $\pm\cdot\mp = -$, and set $-(\pm) = \mp$. Then, for $s, t\in \{+, -\}$, the following fusion rules hold.
        \begin{align}
            V_{1}^s V_{1}^t \cong V_{K\bar K}^s V_{K\bar K}^t&\cong V_{1}^{s\cdot t},& V_{K\bar K}^s V_{1}^t &\cong V_{K\bar K}^{s\cdot t},\label{fusion:SS}\\
            V_K^s V_{1}^t&\cong V_{K}^{s\cdot t}, & V_{\bar K}^s V_{1}^t&\cong V_{\bar K}^{s\cdot t},\label{fusion:SMa}\\ 
            V_K^s V_{K\bar K}^t&\cong V_{\bar K}^{s\cdot t}, & V_K^s V_{K\bar K}^t&\cong V_{\bar K}^{s\cdot t},\label{fusion:SMb}\\
            V_{1}^s P_{1}^t\cong V_{K\bar K}^s P_{K\bar K}^t&\cong P_{1}^{s\cdot t}, & V_{K\bar K}^s P_{1}^t\cong V_{1}^s P_{K\bar K}^t&\cong P_{K\bar K}^{s\cdot t},\label{fusion:SB}\\
            V_K^s V_K^t\cong V_{\bar K}^s V_{\bar K}^t&\cong P_{K\bar K}^{s\cdot t}, & V_K^s V_{\bar K}^t&\cong P_{1}^{s\cdot t},\label{fusion:MM}
        \end{align}
        \begin{align}
            V_{K}^{s} P_1^{t}\cong V_{K}^{s} P_{K\bar K}^{-t}\cong V_{\bar K}^{s} P_1^{-t}\cong V_{\bar K}^{s} P_{K\bar K}^{t}&\cong 2^{2n-1}V_K^{s\cdot t}\oplus 2^{2n-1}V_{\bar K}^{-s\cdot t},\label{fusion:MB}\\
            P_1^s P_{1}^t \cong P_{K\bar K}^s P_{K\bar K}^t\cong  P_1^s P_{K\bar K}^{-t}&\cong 2^{2n-1}P_1^{s\cdot t}\oplus 2^{2n-1}P_{K\bar K}^{-s\cdot t}.\label{fusion:BB}
        \end{align}
    \end{proposition}
    \begin{proof}
        Isomorphisms \eqref{fusion:SS}-\eqref{fusion:MB} almost follow from \cite[Thm. 7.2.5]{modulardata2024} (and the comment after). However, care should be taken in verifying the sign of the modules. For example, in the fourth line, there is an isomorphism $P_{K\bar K}^s\to V_K^+V_K^+$ generated by $(1-K)(1+\sigma\tilde k+\tilde k^2+\sigma\tilde k^3)\mapsto |0\rangle^{\otimes n}\otimes |1\rangle^{\otimes n}$ for some $\sigma\in\{1, -1\}$, where $s=\sgn\sigma$, but $s$ must be determined by the action of $\tilde k$. In this case, since $\tilde k$ is group-like, it acts as 
        $$\tilde k\cdot (|0\rangle^{\otimes n}\otimes |1\rangle^{\otimes n}) = (\tilde k |0\rangle^{\otimes n}\otimes \tilde k |1\rangle^{\otimes n}) = (i|0\rangle^{\otimes n}\otimes (-1)^n i|1\rangle^{\otimes n}) = |0\rangle^{\otimes n}\otimes |1\rangle^{\otimes n}.$$ 
        Thus, $\tilde k$, by the isomorphism, must act by the identity on the generator $(1-K)(1+\sigma\tilde k+\tilde k^2+\sigma\tilde k^3)$ of $P_{K\bar K}^s$, meaning $\sigma=1$ and $s=+$.

        Isomorphisms \eqref{fusion:MB} follows from Proposition \ref{jordan-decomp} and Isomorphisms \eqref{fusion:SMa} and \eqref{fusion:SMb}. Isomorphisms \eqref{fusion:BB} follows from Isomorphisms \eqref{fusion:MM} and \eqref{fusion:MB}.
    \end{proof}

    Propositions \ref{jordan-decomp} and \ref{fusion-rules} give some experimental evidence of the following general conjecture, which is also true when $H$ is ribbon. 
    \begin{conj}\label{tensor-only-one}
        Let $H$ be a finite-dimensional quasitriangular Hopf algebra. Let $M_1, M_2$ be indecomposable $H$-modules and $\tilde M_1, \tilde M_2$ be $\H$-modules for which $\tilde M_1|_H\cong M_1$ and $\tilde M_2|_H\cong M_2$ as $H$-modules. 
        \begin{enumerate}
            \item Suppose the simple $H$-module $V$ is a composition factor of $M_1\otimes M_2$. Then, there is exactly one $\H$-module $\tilde V$ (up to isomorphism) for which $\tilde V|_H\cong V$ and $\tilde V$ is a composition factor of $\tilde M_1\otimes \tilde M_2$.
            \item Suppose the indecomposable $H$-module $M$ is a direct summand of $M_1\otimes M_2$ as $H$-modules. Then, there is exactly one $\H$-module $\tilde M$ (up to isomorphism) for which $\tilde M|_H\cong M$ and $\tilde M$ is a direct summand of $\tilde M_1\otimes \tilde M_2$ as $\H$-modules.
        \end{enumerate}
    \end{conj}

    As a special case of Conjecture \ref{tensor-only-one}, we could set $M_1$ to be the trivial $H$-module in (1) to obtain the following; if the simple $H$-module $V$ is a composition factor of $M$ as $H$-modules, then there is exactly one $\H$-module $\tilde V$ (up to isomorphism) for which $\tilde V|_H\cong V$ and $\tilde V$ is a direct summand of $\tilde M$ as $\H$-modules. This is again supported by Proposition \ref{jordan-decomp}.

    \begin{remark}
        Despite having modules which seem to pair up, neither $D\Kn$ nor $\Kn$ is of the form $\H$ (as Hopf algebras) for some other quasitriangular Hopf algebra $H$. By factorizability, $D\Kn$ for even $n$ cannot be written as $D\Kn\cong \H$ for some $H$. If $\Kn\cong \H$, then by Proposition \ref{simples-double}, $\Rep(H)$ would have precisely one simple object up to isomorphism. This implies $H$ is semisimple. By Corollary \ref{semisimple-iff}, this would imply $\H\cong \Kn$ is semisimple, which is a contradiction.
    \end{remark} 

    \begin{proposition}\label{not-deligne}
        For odd $n$, $\Rep(\RKn)\not\cong \Rep(D\Kn)\boxtimes\Vecm$ as monoidal categories.
    \end{proposition}
    \begin{proof}
         One can verify that such an equivalence cannot simultaneously preserve both Isomorphisms \eqref{fusion:MM} and \eqref{fusion:BB}.
    \end{proof}
    
\section*{Acknowledgements}
The author is supported by the National Science Foundation Graduate Research Fellowship Program under Grant No. 2139319. Any opinions, findings, and conclusions or recommendations expressed in this material are those of the author and do not necessarily reflect the views of the National Science Foundation. 

The author thanks Y. Sommerh\"auser for a helpful discussion on cocycled crossed products. The author also thanks B. R. Jones for illustrating the proof of Lemma \ref{sqrt-uSu} with the holomorphic functional calculus in the complex case. Finally, the author thanks K. Goodearl, Z. Wang, and Q. Zhang for many useful comments.

\bibliographystyle{abbrv}
\bibliography{zbib}

\begin{thebibliography}{10}

\bibitem{Agore2013}
A.~L. Agore.
\newblock Crossed {P}roduct of {H}opf {A}lgebras.
\newblock {\em Communications in Algebra}, 41(7):2519--2542, 2013.

\bibitem{Agore2011extending}
A.~L. Agore and G.~Militaru.
\newblock {Extending structures II: The quantum version}.
\newblock {\em Journal of Algebra}, 336(1):321--341, 2011.

\bibitem{Andruskiewitsch2014hopftensor}
N.~Andruskiewitsch, I.~Angiono, A.~Garc{\'i}a~Iglesias, B.~Torrecillas, and C.~Vay.
\newblock From {H}opf {A}lgebras to {T}ensor {C}ategories.
\newblock In {\em Conformal Field Theories and Tensor Categories}, pages 1--31. Springer Berlin Heidelberg, 2014.

\bibitem{BlattnerCohen1986}
R.~J. Blattner, M.~Cohen, and S.~Montgomery.
\newblock Crossed products and inner actions of {H}opf algebras.
\newblock {\em Transactions of the American Mathematical Society}, 298(2):671--711, 1986.

\bibitem{BruguieresNatale11}
A.~Bruguières and S.~Natale.
\newblock {Exact {S}equences of {T}ensor {C}ategories}.
\newblock {\em International Mathematics Research Notices}, 2011(24):5644--5705, 01 2011.

\bibitem{modulardata2024}
L.~Chang, Q.~T. Kolt, Z.~Wang, and Q.~Zhang.
\newblock Modular data of non-semisimple modular categories.
\newblock arXiv:2404.09314, 2024.

\bibitem{cohen2008characters}
M.~Cohen and S.~Westreich.
\newblock Characters and a {V}erlinde-type formula for symmetric {H}opf algebras.
\newblock {\em Journal of Algebra}, 320(12):4300--4316, 2008.

\bibitem{Costantino:2023bjb}
F.~Costantino, N.~Geer, B.~Ha\"\i{}oun, and B.~Patureau-Mirand.
\newblock {Skein (3+1)-TQFTs from non-semisimple ribbon categories}.
\newblock arXiv:2306.03225, 2023.

\bibitem{Crane1997statesum}
L.~Crane, L.~H. Kauffman, and D.~N. Yetter.
\newblock State-sum invariants of 4-manifolds.
\newblock {\em Journal of Knot Theory and Its Ramifications}, 06(02):177--234, 1997.

\bibitem{DeRenzi2022tqft}
M.~De~Renzi, A.~M. Gainutdinov, N.~Geer, B.~Patureau-Mirand, and I.~Runkel.
\newblock {3-Dimensional TQFTs from non-semisimple modular categories}.
\newblock {\em Selecta Mathematica}, 28(2):42, Jan 2022.

\bibitem{EGNO}
P.~Etingof, S.~Gelaki, D.~Nikshych, and V.~Ostrik.
\newblock {\em Tensor categories}, volume 205 of {\em Mathematical Surveys and Monographs}.
\newblock American Mathematical Society, Providence, RI, 2015.

\bibitem{ENO2004Radford}
P.~Etingof, D.~Nikshych, and V.~Ostrik.
\newblock {An analogue of Radford's $S^4$ formula for finite tensor categories}.
\newblock {\em International Mathematics Research Notices}, 2004(54):2915--2933, 01 2004.

\bibitem{eno2005fusion}
P.~Etingof, D.~Nikshych, and V.~Ostrik.
\newblock On fusion categories.
\newblock {\em Journal of Algebra}, 162:581--642, 2005.

\bibitem{etingof2004finite}
P.~Etingof and V.~Ostrik.
\newblock Finite tensor categories.
\newblock {\em Moscow Mathematical Journal}, 4(3):627--654, 782--783, 2004.

\bibitem{farsad2022symplectic}
V.~Farsad, A.~M. Gainutdinov, and I.~Runkel.
\newblock The symplectic fermion ribbon quasi-{H}opf algebra and the {$\operatorname{SL}(2,\mathbb Z)$}-action on its centre.
\newblock {\em Advances in Mathematics}, 400:108247, 2022.

\bibitem{Henriques2015CategorifiedTF}
A.~Henriques, D.~Penneys, and J.~E. Tener.
\newblock Categorified trace for module tensor categories over braided tensor categories.
\newblock {\em Documenta Mathematica}, pages 1089--1149, 2015.

\bibitem{kauffman1993necessary}
L.~H. Kauffman and D.~E. Radford.
\newblock A necessary and sufficient condition for a finite-dimensional {D}rinfeld double to be a ribbon {H}opf algebra.
\newblock {\em Journal of Algebra}, 159(1):98--114, 1993.

\bibitem{kerler2003homology}
T.~Kerler.
\newblock Homology {TQFT}'s and the {A}lexander-{R}eidemeister invariant of 3-manifolds via {H}opf algebras and skein theory.
\newblock {\em Canadian Journal of Mathematics}, 55(4):766--821, 2003.

\bibitem{panaite1999quasitriangular}
F.~Panaite and F.~Van~Oystaeyen.
\newblock Quasitriangular structures for some pointed {H}opf algebras of dimension $2^n$.
\newblock {\em Commununications in Algebra}, 27(10):4929--4942, 1999.

\bibitem{Reshetikhin1991invariants}
N.~Reshetikhin and V.~G. Turaev.
\newblock Invariants of 3-manifolds via link polynomials and quantum groups.
\newblock {\em Inventiones mathematicae}, 103(1):547--597, Dec 1991.

\bibitem{ReshetikhinTuraev1990}
N.~Y. Reshetikhin and V.~G. Turaev.
\newblock {Ribbon graphs and their invariants derived from quantum groups}.
\newblock {\em Communications in Mathematical Physics}, 127(1):1--26, 1990.

\bibitem{Reutter2023semisimple}
D.~Reutter.
\newblock {Semisimple four-dimensional topological field theories cannot detect exotic smooth structur}e.
\newblock {\em Journal of Topology}, 16(2):542--566, 2023.

\bibitem{shimizu2019non}
K.~Shimizu.
\newblock Non-degeneracy conditions for braided finite tensor categories.
\newblock {\em Advances in Mathematics}, 355:106778, 36, 2019.

\bibitem{turaev1992modular}
V.~G. Turaev.
\newblock Modular categories and 3-manifold invariants.
\newblock {\em International Journal of Modern Physics B}, 06(11n12):1807--1824, 1992.

\end{thebibliography}
\end{document}